\documentclass{amsart}  

\usepackage{amsmath, amsthm, amssymb}
\newtheorem{thm}{Theorem}
\newtheorem{prop}[thm]{Proposition}
\newtheorem{lem}[thm]{Lemma}
\newtheorem{thm*}{Theorem}
\newtheorem{cor}[thm]{Corollary}
\newtheorem*{thmM2}{Theorem \ref{thm:M2}}
\newtheorem*{thmM4}{Theorem \ref{thm:M4}}
\begin{document}

\title[Parabolic constructions]{Parabolic constructions of asymptotically flat 3-metrics of Prescribed
scalar curvature
%Insert your title here
}

% if too long for running head

\author{Chen-Yun Lin}
%First Author         \and        Second Author %etc.

\address{Department of Mathematics
196 Auditorium Road, Unit 3009
Storrs, CT 06269-3009, USA}
\email{E-mail address: chen-yun.lin@uconn.edu}

\maketitle

\begin{abstract}
In 1993, Bartnik \cite{Bar93} introduced a quasi-spherical construction
of metrics of prescribed scalar curvature on 3-manifolds. Under quasi-spherical
ansatz, the problem is converted into the initial value problem for
a semi-linear parabolic equation of the lapse function. The original
ansatz of Bartnik started with a background foliation with round metrics
on the 2-sphere leaves. This has been generalized by several authors
\cite{ST02,SW04,Smi09} under various assumptions on the background
foliation. In this article, we consider background foliations given
by conformal round metrics, and by the Ricci flow on 2-spheres. We
discuss conditions on the scalar curvature function and on the foliation
that guarantee the solvability of the parabolic equation, and thus
the existence of asymptotically flat 3-metrics with a prescribed inner
boundary. In particular, many examples of asymptotically flat-scalar
flat 3-metrics with outermost minimal surfaces are obtained.
%\keywords{parabolic construction \and asymptotically flat 3-metrics \and prescribed scalar curvature}
%\subclass{MSC 83 \and MSC 53}
\end{abstract}

%%%% Introduction

\section{Introduction}
\label{intro}
Einstein's field equation of a space-time $\left(V,\gamma\right)$
is \[R_{ab}^{V}-\frac{1}{2}R^{V}\gamma_{ab}=8\pi T_{ab,}\quad a,b=0,1,2,3,\]
where $T_{ab}$ is the space-time energy momentum tensor. The equation
admits Cauchy data formulation and initial data cannot be chosen arbitrarily.
Let $\left(V,\gamma\right)$ be a solution and consider a space-like
hypersurface $\left(N,g\right)$. From the Gauss and Codazzi equations
the scalar curvature $\bar{R}$ and the second fundamental form $k$
of $\left(N,g\right)$ will satisfy the following constraint equations \cite{Wal84}:
\begin{eqnarray*}
\bar{R}+\left(tr_{g}k\right)^{2}-|k|_{g}^{2} & = & 16\pi T_{00},\\
\nabla_{j}\left(k_{ij}-trkg_{ij}\right) & = & 8\pi T_{0i}\quad i,j=1,2,3,
\end{eqnarray*}
where $e_{0}$ is the future time-like unit normal vector of the
hypersurface $N.$ When $T=0$, these equations are called the vacuum
constraint equations. There are various ways to construct solutions
of the constraint equations. In 1993 Bartnik \cite{Bar93} introduced
a new construction of 3-metrics of prescribed scalar curvature and
prescribed the inner boundary using a quasi-spherical ansatz. A manifold
$N$ is called quasi-spherical if it can be foliated by round spheres.
Under quasi-spherical ansatz, the equation for the scalar curvature
$\bar{R}$ can be written as a semilinear parabolic equation. Let $\Sigma$
be a smooth compact surface without boundary. Let $N=[a,\infty)\times\Sigma$
be equipped with a quasi-spherical metric 
\[
\bar{g}=u^{2}dt^{2}+{\displaystyle \sum_{i=1}^{2}\left(\beta_{i}dt+t\sigma_{i}\right)^{2}}
\]
for some functions $u$ and $\beta_{i},i=1,2,$ where $\sigma_{i}^{2}$
is the standard metric on the unit 2-sphere. By viewing $u$ as an
unknown function and the scalar curvature $\bar{R}$ of $\left(N,\bar{g}\right)$
and $\beta_{i}$ as prescribed fields, Bartnik \cite{Bar93} observed
that the function $u$ is given by 
\[
2t\frac{\partial u}{\partial t}-2\beta_{i}u_{|i}
=
\gamma u^{2}\Delta u+\left(1+\gamma B\right)u-\gamma\left(1-\frac{1}{2}\bar{R}t^{2}\right)u^{3},
\]
where $\cdot_{|i}=\nabla_{i}\cdot$ denotes the covariant derivative
of $d\sigma^{2}$, $\gamma=\left(1-\frac{1}{2}div\beta\right)^{-1}$,
and 
$B=\frac{1}{2}\left|div\beta\right|^{2}+\frac{1}{8}\left|\beta_{i|j}+\beta_{j|i}\right|^{2}-t\partial_{t}\left(div\beta\right)+\beta_{i}\left(div\beta\right)_{|i}-\frac{3}{2}div\beta$. 

Bartnik's parabolic method under quasi-spherical ansatz has been generalized
by several authors. In 2002, Shi and Tam \cite{ST02} used the foliation
by level sets of the distance function to a convex hypersurface in
$\mathbb{R}^{n}$. Let $\Sigma_{0}$ be a smooth compact strictly
convex hypersurface in $\mathbb{R}^{n}$, and $t$ the distance function
from $\Sigma_{0}$. The metric $\bar{g}$ on $N=[a,\infty)\times\Sigma$ is of the form 
\[
\bar{g}=u^{2}dt^{2}+g_{t},
\]
where $g_{t}$ is the induced metric on $\Sigma_{t}$, which is the
hypersurface with distance $t$ from $\Sigma_{0}.$ The function $u$
with prescribed flat-scalar curvature $\bar{R}=0$ satisfies the equation
\[
2H_{0}\frac{\partial u}{\partial t}=2u^{2}\Delta_{t}u+\left(u-u^{3}\right)R_{t},
\]
where $H_{0}$ is the mean curvature of $\Sigma_{t}$ in $\mathbb{R}^{n}$,
$R_{t}$ is the scalar curvature of $\Sigma_{t}$, and $\Delta_{t}$
is the Laplacian on $\Sigma_{t}.$ Shi and Tam showed that for a smooth
compact strictly convex hypersurface $\Sigma_{0}$ in $\mathbb{R}^{n}$,
a positive function $u$ can be arbitrarily prescribed initially. 

Weinstein and Smith studied this parabolic method under quasi-convex
foliations in \cite{SW00,SW04}. A topological 2-sphere is said to
be quasi-convex if its Gauss and mean curvature are positive. A foliation
is quasi-convex if its leaves are quasi-convex spheres. Using the
Poincar\'{e} Uniformization, $\bar{g}$ on $N=[a,\infty)\times\Sigma$ can
be written as 
\[
\bar{g}=u^{2}dt^{2}+e^{2v}g_{ij}\left(\hat{\beta}^{i}dt+td\theta^{i}\right)\left(\hat{\beta}^{j}dt+td\theta^{j}\right),
\]
where $\left(\theta^{1},\theta^{2}\right)$ are local coordinates
on a topological 2-sphere $\Sigma$ and $g_{ij}$ is a fixed round
metric of area $4\pi$. The parabolic equation for $u$ on $\Sigma$
is given by
\[
t\frac{\partial u}{\partial t}-\beta\cdot\nabla u=\Gamma u^{2}\Delta u+Au-Bu^{3},
\]
where $\beta=e^{2v}\hat{\beta}^{i}\partial_{i}$, and $\Gamma$, $A$, and
$B$ are functions that can be calculated in terms of only $\beta,v,$
and $t.$ Weinstein and Smith derived conditions on the source functions
$\beta,\Gamma,A$, and $B$ from above that guarantee the existence
of a global positive solution on the interval $[a,\infty)$. However,
sometimes the decay conditions may not be verified. 

For initial data with an apparent horizon, the parabolicity of the
parabolic equation breaks down on the horizon. To overcome this, Smith
\cite{Smi09} considered the metrics of the form 
\[
\bar{g}=\frac{u^{2}}{1-\frac{a}{t}}dt^{2}+t^{2}g(t),
\]
on $N=[a,\infty)\times\Sigma$, where $\left(\Sigma,h\right)$ is a given Riemannian surface, $g(t)$
satisfies 
$$a^{2}g(a)=h, \quad {\displaystyle \frac{\partial g_{ij}}{\partial t}g^{ij}>-4},$$ 
$$g(t)=g(a)  \mbox{ for } t\in[a,a+\epsilon),$$ 
$$g(t)=\sigma  \mbox{ for $t$ large},$$
where $\sigma$ is the standard metric on $S^2,$
and the scalar curvature $R(t)$ of $g(t)$ is positive. The second condition,
$g(t)=g(a) \mbox{ for } t\in[a,a+\epsilon)$, allows the separation of variables
so that solving the parabolic scalar curvature equation reduces to
solving the elliptic equation
\[
\Delta_{g}u-\frac{R}{2}u+\frac{1}{u}=0
\]
on the region $[a,a+\epsilon]\times\Sigma,$ which is referred as
the collar region. He then obtained asymptotically flat time symmetric
initial data on $N$ by constructing the metric on the collar region,
and the metric exterior to the collar region using the parabolic method. 

This parabolic construction also provides an insight into the extension
problem, which is suggested by the definition of quasi-local mass
\cite{Bar89}. Here one hopes to extend a bounded Riemannian 3-manifold
$\left(\Omega,g_{0}\right)$ to an asymptotically flat 3-manifold
$\left(M,g\right)$ with nonnegative scalar curvature containing $\left(\Omega,g_{0}\right)$
isometrically. The condition that the scalar curvature can be defined
distributionally and bounded across $\partial(M\setminus\Omega)$
leads to the geometric boundary conditions
\[
g|_{\partial(M\setminus\Omega)}=g_{0}|_{\partial\Omega},\quad H_{\partial(M\setminus\Omega),g}=H_{\partial\Omega,g_{0}}
\]
where the metrics and the mean curvatures in $\left(\Omega,g_{0}\right)$
and $\left(M,g\right)$ match along the boundary $\partial(M\setminus\Omega)$.
Note that this extension is not possible with the traditional conformal
method \cite{CBY,SY79}. Specifying both the boundary metric and the
mean curvature leads to simultaneous Dirichlet and Neumann boundary
conditions, which are ill-posed for the elliptic equation of the conformal
factor.

In this paper, we let $e^{2f}\sigma$ be a family of metrics on
a topological sphere $\Sigma$ where $\sigma$ is the standard metric on the sphere.  Let $N=[1,\infty)\times\Sigma$, $\bar{R}$
be a given function on $N$, and 
$$\bar{g}=u^{2}dt^{2}+t^{2}e^{2f}\sigma$$
be a metric on $N$ with scalar curvature $\bar{R}$. The metric $\bar{g}$ has the
scalar curvature $\bar{R}$ if and only if $u$ satisfies the parabolic equation
\begin{eqnarray}
\left(\frac{1}{t}+\frac{\partial f}{\partial t}\right)\frac{\partial u}{\partial t} 
& = & \frac{1}{2t^{2}}u^{2}\Delta_{f}u+\left(\frac{\partial}{\partial t}\left(\frac{1}{t}+\frac{\partial f}{\partial t}\right)+\frac{3}{2}\left(\frac{1}{t}+\frac{\partial f}{\partial t}\right)^{2}\right)u\label{eq:u}\\
&  & -\frac{1}{4t^{2}}\left(R_{f}-t^{2}\bar{R}\right)u^{3},\nonumber 
\end{eqnarray}
where $R_f$ and $\Delta_f$ denote the scalar curvature and the Laplacian with respect to $e^{2f}\sigma$, respectively.
 
First, we study decay conditions of the foliation and prescribed scalar
curvature which ensure the solution $u$ gives an asymptotically flat
metric with prescribed scalar curvature (Theorem \ref{thm:M1}). Second,
with suitable decay conditions we show that there exists a solution
$u^{-1}\in C^{2+\alpha}\left(N\right)$ so that the metric $\bar{g}$
has outermost totally geodesic boundary. Instead of assuming a collar
region, we use the dilation invariance of weighted H\"{o}lder norms together
with suitable curvature conditions to obtain uniform bounds of solutions
to the initial value problem (\ref{eq:u}) with initial condition
$u^{-1}\left(1+\epsilon,\cdot\right)$ on $[1+\epsilon,\infty)$.
By Arzela-Ascoli Theorem, there exists a weak solution to (\ref{eq:u})
with $u^{-1}\left(1,\cdot\right)=0$ (Theorem \ref{thm:M2}). Since
the mean curvature ${\displaystyle H=\frac{2}{u}\left(\frac{1}{t}+\frac{\partial f}{\partial t}\right)}$
of $\Sigma_{t}=\{ t \} \times \Sigma$ stays positive, by the maximum principle the boundary
surface is the outermost minimal surface. Theorem \ref{thm:M1} and
Theorem \ref{thm:M2} together give an initial data set of prescribed
geometry with a horizon. Last, we study existence results under Ricci
flow foliation. It is known by the work of R. Hamilton \cite{Ham88}
and B. Chow \cite{Cho91} that the evolution under Ricci flow of an
arbitrary initial metric on a topological 2-sphere, suitably normalized,
exists for all time and converges exponentially to a round metric.
Given a compact Riemannian surface $\left(\Sigma,g_{1}\right)$, let
\[
\bar{g}=u^{2}dt^{2}+t^{2}g(t)\]
be  a metric on $N=[1,\infty) \times \Sigma$ where $g(t)$ evolves by the Ricci flow defined in (\ref{eq:RF}).
Using the fast convergence property, we have corresponding existence results
(Theorem \ref{thm:M3} and Theorem \ref{thm:M4}) under Ricci flow
foliation. Since $\Sigma_{t}$ are nearly round, the ADM mass of the
asymptotically flat manifold $N$ is approached by the Hawking mass
$m_{H}\left(\Sigma_{t}\right)$ as $t\rightarrow\infty$ (see \cite{SWW09}).
If in addition the prescribed scalar  curvature $\bar{R}$  is nonnegative, by
the equation (\ref{eq:RFw}) of $w=u^{-2}$, a direct computation shows
that 
\[
\frac{d}{dt}m_{H}(\Sigma_{t})=\frac{1}{8\pi}\int w|\nabla u|^{2}+\frac{t^{2}}{2}|M_{ij}|^{2}w+\frac{t^{2}}{2}\bar{R}d\sigma\geq0,
\]
where $M_{ij}$ is the trace-free part of the Ricci potential. We
obtain an interesting byproduct, the Hawking mass $m_{H}(\Sigma_{t})$
is nondecreasing in $t$. If we impose the initial condition $u(1,\cdot)^{-1}=0$,
i.e., minimal boundary surface, and assume that $\bar{R}\geq0$, then
ADM mass is bounded below by\[
{\displaystyle m_{ADM}(N)\geq\frac{1}{2}\sqrt{\frac{A(\Sigma)}{4\pi}}=\frac{1}{2}.}\]
In particular, if we start with the standard metric $(\Sigma,\sigma)$
and prescribe flat scalar curvature $\bar{R}\equiv0$, then the metric
$\bar{g}$ obtained from above would be exactly a Schwarzschild metric
with ADM mass ${\displaystyle m_{ADM}=\frac{1}{2}}$ (Corollary \ref{cor:schwarzchild}).

Let $A_{I}=I\times\Sigma$
for any interval $I\subset[1,\infty)$. For compact intervals $I\subset[1,\infty),$
the parabolic H\"{o}lder space $C^{k+\alpha}(A_{I})$ is the Banach space
of continuous functions on $A_{I}$ with finite weighted $\left\| \cdot\right\| _{k+\alpha,I}$
norm, and for $I$ noncompact, $C^{k+\alpha}(A_{I})$ is defined as
the space of continuous functions which are norm-bounded on compact
subsets of $I.$ $C^{k,\alpha}(\Sigma)$ is the H\"{o}lder space on $\Sigma$
with norm $\left\| \cdot\right\| _{k,\alpha}.$ For any $\xi\in C^{0}(A_{I})$,
we define functions $\xi_{*}(t)$ and $\xi^{*}(t)$ by
\[
\xi_{*}(t)=\inf\{\xi(t,x):x\in\Sigma\},\quad\xi^{*}(t)=\sup\{\xi(t,x):x\in\Sigma\}.
\]
Our main theorems in this paper are the following:

%%%%%% Main Theorem 1

\begin{thm} \label{thm:M1}
Assume $\bar{R}\in C^{\alpha}(N)$ and $f\in C^{4+\alpha}(N)$ such that 
$$   
    0 < 1+t\frac{\partial f}{\partial t} <\infty \mbox{ for all } 1\leq t\leq \infty ,
$$
$$
{\displaystyle \left(t\frac{\partial f}{\partial t}\right)^{*}\in L^{1}\left([1,\infty)\right)\,},
$$
$$
t\left(\frac{\partial}{\partial t}\ln\left(\frac{1}{t}+\frac{\partial f}{\partial t}\right)\right)^{*}-t\frac{d}{dt}\ln\left(\frac{1}{t}+\frac{\partial f}{\partial t}\right)^{*}\in L^{1}\left([1,\infty)\right),
$$
and 
$$
\int_{1}^{\infty}|R_{f}-2|^{*}+|\bar{R}t^{2}|^{*}dt<\infty.$$
Suppose there exists a constant  $C>0$ such that for all $t\geq2$ and $I_{t}=[t/2,2t]$, 
$$
{\displaystyle \left\| \bar{R}t^{2}\right\| _{\alpha,I_{t}}+  \left\| t\frac{\partial f}{\partial t}\right\| _{2+\alpha,I_{t}}  + \left\|1-e^{-2f}\right\|_{2+\alpha,I_{t}}  \leq\frac{C}{t} }.
$$
Further assume the  nonnegative constant $K$ defined by
 \[
   K=\sup_{1\leq t<\infty}
               \left\{ 
                        -\int_{1}^{t}
                                \frac{1}{2t^{2}}
                                \left(\frac{R_{f}-t^{2}\bar{R}}{\frac{1}{\tau}+\frac{\partial f}{\partial\tau}}\right)_*
                                e^{\int_{1}^{\tau}\left(2\frac{\partial}{\partial s}\ln\left(\frac{1}{s}+\frac{\partial f}{\partial s}\right)+3\left(\frac{1}{s}+\frac{\partial f}{\partial s}\right)\right)^*ds}
                         d\tau
\right\} 
\]
satisfies 
\[
K<\infty.
\]
Then for every $\varphi(x)\in C^{2,\alpha}\left(\Sigma\right)$
satisfying 
\[
0<\varphi(x)<\frac{1}{\sqrt{K}}\quad\hbox{for\, all\,}\, x\in\Sigma,
\]
there is a unique positive solution $u\in C^{2+\alpha}\left(N\right)$
of (\ref{eq:u}) with the initial condition 
\begin{equation}
u(1,\cdot)=\varphi(\cdot)\label{eq:initial}
\end{equation}
such that the metric $\bar{g}=u^{2}dt^{2}+t^{2}e^{2f}\sigma$ on
$N$ satisfies the asymptotically flat condition 
\begin{equation}
\left|\bar{g}_{ab}-\delta_{ab}\right|+t\left|\partial_{a}\bar{g}_{bc}\right|<\frac{C}{t},\quad a,b,c=1,2,3\label{AF condition}
\end{equation}
with ADM mass of $\left(N,\bar{g}\right)$ that can be expressed as
         \begin{eqnarray*}
                      m_{ADM} = \frac{1}{4\pi}\lim_{t\rightarrow\infty}\oint_{S_{t}}\frac{t}{2}\left(1-u^{-2}\right)d\sigma.
          \end{eqnarray*}
Moreover the Riemannian curvature $\bar{Rm}$ of the 3-metric
$\bar{g}$ on $N$ is H\"{o}lder continuous and decays as ${\displaystyle \left|\bar{Rm}\right|<\frac{C}{t^{3}}}$.
%%%%
\end{thm}

%%%%%% Main Theorem 2

\begin{thm}\label{thm:M2}
%%%%%
Let $t\frac{\partial f}{\partial t}  \in C^{2+\alpha}\left(N\right)$ and $\bar{R} \in  C^{\alpha}\left(N\right)$ be given such that 
$$
                 \delta_{*}(t) 
                           =
                                    \int_{1}^{t}   \frac{1}{2\tau^{2}}
                                              \left(
                                                        \frac{R_f-\tau^2\bar{R}}{\frac{1}{\tau}+\frac{\partial f}{\partial\tau}}
                                               \right)_{*}
                                    e^{-\int_{\tau}^{t}
                                               \left(
                                                         2\frac{\partial}{\partial s} \ln
                                                                     \left(
                                                                               \frac{1}{s}+\frac{\partial f}{\partial s}
                                                                      \right)
                                                          +3
                                                                      \left(
                                                                                   \frac{1}{s}+\frac{\partial f}{\partial s}
                                                                      \right)
                                                \right)^{*}ds}
                                       d\tau,
$$
and
$$
\delta^{*}(t) = \int_{1}^{t}\frac{1}{2\tau^{2}}\left(\frac{R_{f}-\tau^{2}\bar{R}}{\frac{1}{\tau}+\frac{\partial f}{\partial\tau}}\right)^{*}e^{-\int_{\tau}^{t}\left(2\frac{\partial}{\partial s}\ln\left(\frac{1}{s}+\frac{\partial f}{\partial s}\right)+3\left(\frac{1}{s}+\frac{\partial f}{\partial s}\right)\right)_{*}ds}d\tau
$$
are finite on $[1,\infty)$. Further suppose that for all $1\leq t<\infty$, 
\[
0<1+t\frac{\partial f}{\partial t} <\infty,
\]
and 
\[
t^{2}\bar{R}<R_{f}.
\]
Then there is $u^{-1}\in C^{2+\alpha}\left(A_{(1,\infty)}\right)$
such that the metric $\bar{g}$ on $N$ has curvature uniformly bounded
on $A_{[1,2]}$ with totally geodesic boundary. 

Let $0<\eta<1$ be
such that 
\begin{equation}
1-\eta<R_{f}-t^{2}\bar{R}|_{t=1}<\left(1-\eta\right)^{-1}.\label{bdy est}
\end{equation}
Then there is $t_{0}>1$ such that $1<t<t_{0},$
\[
\frac{t-1}{t}\left(1-\eta\right)<u^{-2}\left(t\right)<\frac{t-1}{t}\left(1-\eta\right)^{-1},
\]
which gives 
\[
1-\frac{\eta}{1-\eta}\left(t-1\right)\leq2m\leq1+\eta\left(t-1\right).
\]
\end{thm}

The modified Ricci flow is defined by the geometric evolution equation
\begin{equation}
     \left\{
             \begin{array}{cc}
                 \frac{\partial}{\partial t}g_{ij}=\left(r-R\right)g_{ij}+2D_{i}D_{j}\mathcal{F}=2M_{ij}\\
                     g\left(1,\cdot\right)=g_{1}\left(\cdot\right),
             \end{array}
     \right.\label{eq:RF}
\end{equation}
where $R$ is the scalar curvature, $r=\int Rd\mu/\int1d\mu$ is the
mean scalar curvature, and the Ricci potential $ \mathcal{F}$ is a solution of the
equation $\Delta \mathcal{F}=R-r$ with mean value zero. The equation of $ \mathcal{F}$
satisfies \[
\frac{\partial  \mathcal{F}}{\partial t}=\Delta  \mathcal{F}+r \mathcal{F}-\int|D \mathcal{F}|^{2}d\mu/\int1d\mu.\]
$M_{ij}=\left(r-R\right)g_{ij}/2+D_{i}D_{j} \mathcal{F}$ is the trace-free part
of $\mathrm{Hess}\left( \mathcal{F}\right)$. The solution under the modified Ricci flow
of an arbitrary initial metric on a topological 2-sphere $\Sigma$
exists for all time and converges exponentially to the round metric,
and $|M|_{ij}\rightarrow0$ exponentially \cite{Cho91,Ham88}. Moreover
if $R\geq0$ at the start, it remains so for all time. The modified
Ricci flow also preserves area. For convenience we normalize the area
so that $A(\Sigma)=4\pi$. By the Gauss-Bonnet formula the mean scalar
curvature $r=\int Rd\mu/\int1d\mu=2$. The solution to the modified Ricci flow provides a canonical
foliation on $N=[1,\infty)\times\Sigma$. 

Let  $\bar{g}=u^{2}dt^{2}+t^{2}g(t)$ be a $3$-metric on $N=[1,\infty)\times\Sigma$ 
and  $g(t)$ be the solution of the Ricci flow defined by (\ref{eq:RF}). Let $R$ and $\bar{R}$ be the scalar curvatures of $g(t)$ and $\bar{g}$, respectively.
Theorems \ref{thm:M3} and Theorem \ref{thm:M4} are existence results
analogue to Theorems \ref{thm:M1} and \ref{thm:M2}, but under Ricci
flow ansatz.

%%%%%% Main Theorem 3

\begin{thm}\label{thm:M3}
Assume that $\bar{R}\in C^{\alpha}(N)$ and the constant
$K$ is defined by
\[
{\displaystyle K=\sup_{1\leq t<\infty}\left\{ -\int_{1}^{t}\left(\frac{R}{2}-\frac{\tau^{2}}{2}\bar{R}\right)_{*}\exp\left(\int_{1}^{\tau}\frac{s{|M|^{*}}^{2}}{2}ds\right)d\tau\right\} <\infty.}
\]
Suppose there is a constant $C>0$ such that for all $t\geq1$ and
$I_{t}=[t,4t]$,
\[
\left\| \bar{R}t^{2}\right\| _{\alpha,I_{t}}\leq\frac{C}{t},\quad\hbox{and}\quad\int_{1}^{\infty}|\bar{R}|^{*}t^{2}dt<\infty.
\]
Then for any function $\varphi\in C^{2,\alpha}(\Sigma)$ satisfying
\[
\displaystyle 0<\varphi<\frac{1}{\sqrt{K}},
\]
there is a unique positive solution $u\in C^{2+\alpha}(N)$ of the
parabolic equation 
\begin{equation*}
t\frac{\partial u}{\partial t}=\frac{1}{2}u^{2}\Delta u+\frac{t^{2}}{4}\left|M\right|^{2}u+\frac{1}{2}u-\frac{1}{4}\left(R-t^{2}\bar{R}\right)u^{3}
\end{equation*}
with initial condition $u(1,\cdot)=\varphi(\cdot)$ such that the
metric $\bar{g}=u^{2}dt^{2}+t^{2}g(t)$ on $N$ satisfies the asymptotically
flat condition with finite ADM mass and
\[
{\displaystyle m_{ADM}=\lim_{t\rightarrow\infty}\frac{1}{4\pi}\oint_{\Sigma_{t}}\frac{t}{2}(1-u^{-2})d\sigma.}
\]
Moreover, the Riemannian curvature $\bar{Rm}$ of the 3-metric $\bar{g}$
on $N$ is H\"{o}lder continuous and decays as ${\displaystyle |\bar{Rm}|<\frac{C}{t^{3}}}$. 
\end{thm}

%%%%%% Main Theorem 4

\begin{thm}
\label{thm:M4}
 Let $\bar{R}\in C^{\alpha}(N).$ Suppose that $\bar{R}t^{2}<R$
for $1\leq t<\infty$. Let $0<\eta<1$ be such that\begin{equation*}
1-\eta< R-\bar{R}|_{t=1}<(1-\eta)^{-1}.\end{equation*}
Then there is a solution $u^{-1}\in C^{2+\alpha}(N)$ such that the
constructed metric on $N$ has curvature uniformly bounded on $A_{[1,2]}$
with totally geodesic boundary $\Sigma$.
\end{thm}

The outline of the paper is as follows: In Section 2, we derive the
parabolic equation for $u$ and its equivalent forms. In Section 3,
we prove Theorem \ref{thm:M1} in two steps. First, we prove the existence
of $u$ (Theorem \ref{thm:existence}). Second, we discuss decay conditions
for $\bar{R}$ and metrics $e^{2f}\sigma$ which ensure asymptotic
flatness of $\bar{g}$ (Theorem \ref{thm:AF}). After we prove Theorem \ref{thm:M1}, we show there exists
a solution $u$ such that the boundary surface is the outermost minimal
surface (Theorem \ref{thm:M2}). In Section 4, we prove similar existence
results under Ricci flow foliations.

%%%%%%%
%Section 2---Curvature Calculations
%%%%%%% 

\section{Curvature calculations}
\label{sec:1}
For now we use $g(t)$ to denote a family of metrics on $\Sigma$.
Let   $\bar{g}=u^{2}dt^{2}+g(t)$ be a metric
on $N=[1,\infty)\times\Sigma$, and $\bar{R}$ and $R$ denote the scalar
curvatures of $\bar{g}(t)$ and $g$ respectively. Let $H$ and $\left|A\right|$
denote the mean curvature and the norm squared of the second fundamental
form of $\Sigma_t=\{ t \} \times \Sigma$. A direct computation shows that the Ricci curvature
of $g$ is given by
\[
g^{ij}R_{3ij}^3=-\frac{1}{u}\frac{\partial}{\partial t}H-\frac{1}{u}\Delta u-|A|^2,
\]
where  $\Delta$ is the Laplacian
with respect to $g(t)$.
The Gauss equation gives that 
\[
g^{ik}g^{jl}\bar{R}_{ijkl}=R-H^{2}+|A|^{2}\mbox{ where }i,j,k,l=1,2.
\]
Combining the above two equations, the scalar curvature $\bar{R}$ with
a metric of the form 
\[
\bar{g}=u^{2}dt^{2}+g(t)
\]
is given by 
\begin{equation}
\bar{R}=-\frac{2}{u}\frac{\partial}{\partial t}H-\frac{2}{u}\Delta u-|A|^{2}+R-H^{2}.\label{eq:R}
\end{equation}
The second fundamental forms $h_{ij},i,j=1,2$ measure the change of the metric along
the normal direction $\nu$. 

When $g(t)=t^{2}e^{2f}\sigma,$
the second fundamental forms are 
\[
h_{ij}=\frac{1}{u}\left(\frac{1}{t}+\frac{\partial f}{\partial t}\right)\sigma_{ij},\mbox{ where }i,j=1,2.
\]
In particular, 
\[
H=\frac{2}{u}\left(\frac{1}{t}+\frac{\partial f}{\partial t}\right)\quad\hbox{and}\quad|A|^{2}=\frac{2}{u^{2}}\left(\frac{1}{t}+\frac{\partial f}{\partial t}\right)^{2}.
\]
From (\ref{eq:R}) and above, we have
\[
\bar{R}=\frac{4}{u^{3}}\left(\frac{1}{t}+\frac{\partial f}{\partial t}\right)\frac{\partial u}{\partial t}-\frac{4}{u^{2}}\frac{\partial}{\partial t}\left(\frac{1}{t}+\frac{\partial f}{\partial t}\right)-\frac{2}{t^{2}u}\Delta_{f}u+\frac{1}{t^{2}}R_{f}-\frac{6}{u^{2}}\left(\frac{1}{t}+\frac{\partial f}{\partial t}\right)^{2},
\]
where $\Delta_f$ is the Laplacian with respect to $e^{2f}\sigma$.
We can rewrite it as
\begin{eqnarray*}
\left(\frac{1}{t}+\frac{\partial f}{\partial t}\right)\frac{\partial u}{\partial t} & = & \frac{1}{2t^{2}}u^{2}\Delta_{f}u+\left(\frac{\partial}{\partial t}\left(\frac{1}{t}+\frac{\partial f}{\partial t}\right)+\frac{3}{2}\left(\frac{1}{t}+\frac{\partial f}{\partial t}\right)^{2}\right)u\\
 &  & -\frac{1}{4t^{2}}\left(R_{f}-t^{2}\bar{R}\right)u^{3}.
\end{eqnarray*}
Introducing $w=u^{-2}$ and ${\displaystyle m=\frac{t}{2}\left(1-u^{-2}\right)}$, we have the equivalent forms
\begin{eqnarray}
\left(\frac{1}{t}+\frac{\partial f}{\partial t}\right)\frac{\partial w}{\partial t} 
& = & \frac{1}{2t^{2}w}\Delta_{f}w+\frac{3}{2t^{2}}u\nabla u\cdot\nabla w + \frac{1}{2t^{2}}\left(R_{f}-t^{2}\bar{R}\right)  \nonumber \\
 &  & -\left(2\frac{\partial}{\partial t}\left(\frac{1}{t}+\frac{\partial f}{\partial t}\right)+3\left(\frac{1}{t}+\frac{\partial f}{\partial t}\right)^{2}\right)w,\label{eq:w}
\end{eqnarray} 
and
\begin{eqnarray} 
\left(\frac{1}{t}+\frac{\partial f}{\partial t}\right)\frac{\partial m}{\partial t} 
& = & \frac{u^2}{2t^2}\Delta_{f}m+\frac{3u}{2t^{2}}\nabla u\cdot\nabla m-\left(2\frac{\partial^{2}f}{\partial t^{2}}+\frac{5}{t}\frac{\partial f}{\partial t}+3\left(\frac{\partial f}{\partial t}\right)^{2}\right)m\nonumber \\
 &  & -\frac{1}{4t}\left(R_{f}-2-t^{2}\bar{R}-4t^{2}\frac{\partial^{2}f}{\partial t^{2}}-12t\frac{\partial f}{\partial t}-6\left(t\frac{\partial f}{\partial t}\right)^{2}\right),\label{eq:m}
\end{eqnarray}
where $\Delta_f$ and $\nabla$ are the Laplacian and the covariant derivatives on $\Sigma_t =\{t\}\times \Sigma$ with respect to $e^{2f}\sigma$.

For Ricci flow ansatz, we consider $N=[1,\infty)\times\Sigma$
equipped with the metric \[
\bar{g}=u^{2}dt^{2}+t^{2}g_{ij}(t,x)dx^{i}dx^{j},\] 
 where $g(t,x)$
is the solution to the modified Ricci flow (\ref{eq:RF}). 
Direct computation shows that the second fundamental forms $h_{ij}$ on $\Sigma_{t}$ with respect to the normal ${\displaystyle \nu=\frac{1}{u}\frac{\partial}{\partial t}}$
are given by
\[
        h_{ij}=\frac{1}{u}\left(\frac{1}{t}\bar{g}_{ij}+t^{2}M_{ij}\right),\quad i,j=1,2;
\]
the mean curvature and the norm squared of the second fundamental
form $|A|$ are 
\begin{equation}
        H=\frac{2}{tu}\quad\hbox{and}\quad|A|^{2}=\frac{2}{t^{2}u^{2}}+\frac{\left|M_{ij}\right|^{2}}{u^{2}},\label{eq:RFH}
\end{equation}
respectively. By equation (\ref{eq:R}), the scalar curvature $\bar{R}$
of $\bar{g}$ is given by
     \[
         \bar{R}=\frac{4}{tu^{3}}\frac{\partial u}{\partial t}-\frac{2}{t^{2}u}\Delta u+\frac{R}{t^{2}}-\frac{2}{t^{2}u^{2}}-\frac{\left|M\right|{}^{2}}{u^{2}},
     \]
where $\Delta$ is the Laplacian with respect to $g(t)$.     
The metric $\bar{g}=u^{2}dt^{2}+t^{2}g(t,x)$ has the scalar curvature
$\bar{R}$ if and only if $u$ satisfies the parabolic equation
\begin{eqnarray}
              t\frac{\partial u}{\partial t}=\frac{1}{2}u^{2}\Delta u+\left(\frac{t^{2}}{4}\left|M\right|^{2}+\frac{1}{2}\right)u-\frac{1}{4}\left(R-t^{2}\bar{R}\right)u^{3}.\label{eq:RFu}
\end{eqnarray}
The equations for the corresponding terms $w=u^{-2}$ and ${\displaystyle m=\frac{t}{2}(1-u^{-2})}$
are as follows 
     \begin{equation}
                  t\partial_{t}w=\frac{1}{2w}\Delta w+\frac{3}{2} u \nabla u \cdot \nabla w-\left(\frac{t^{2}}{2}|M|^{2}+1\right)w+\frac{R}{2}-\frac{t^{2}}{2}\bar{R},\label{eq:RFw}
     \end{equation}
and
    \begin{equation}
      t\partial_{t}m=\frac{1}{2}u^{2}\Delta m+\frac{3}{2}u\nabla u \cdot \nabla m-\frac{t^{2}}{2}|M|^{2}m+\frac{t^{3}}{4}|M|^{2}+\frac{t}{2}-\frac{tR}{4}+\frac{t^{3}}{4}\bar{R},\label{eq:RFm}
    \end{equation}
where $\Delta$ and $\nabla$ are the Laplacian and the covariant derivatives on $\Sigma_t =\{t\}\times \Sigma$ with respect to $g(t)$.

%%%%%
%Section 3---Existence
%%%%%

\section{Existence}
\label{sec:2}
In this section we use the maximum principle to obtain $C^{0}$ estimates
for $u^{-2}$ (Proposition \ref{prop: C0bounds}), and prove Schauder
estimates for $u$ and $m$ (Proposition \ref{prop:interior est}).
Using these a priori estimates, we prove long-time existence of solution
$u.$ Then we show that under suitable decay conditions of the foliation
and the scalar curvature $\bar{R}$, the metric $\bar{g}=u^{2}dt^{2}+t^{2}g(t,x)$
is asymptotically flat and has finite ADM mass. We basically follow the argument in \cite{Bar93}, see also \cite{ST02}.

%%%% Proposition 5

\begin{prop}\label{prop: C0bounds}
Suppose $u\in C^{2+\alpha}\left(A_{[t_{0},t_{1}]}\right)$
, $1\leq t_{0}<t_{1}$ is a positive solution to (\ref{eq:u}). Then
for $t_{0}\leq t\leq t_{1}$ we have 
   \begin{eqnarray*}
              u^{-2}\left(t,x\right) 
               & \geq &
                     \int_{t_{0}}^{t}\frac{1}{2\tau^{2}}\left(\frac{R_{f}-\tau^{2}\bar{R}}{\frac{1}{\tau}
                       +\frac{\partial f}{\partial\tau}}\right)_{*}e^{-\int_{\tau}^{t}\left(2\frac{\partial}{\partial s}\ln\left(\frac{1}{s}+\frac{\partial f}{\partial s}\right)+3\left(\frac{1}{s}+\frac{\partial f}{\partial s}\right)\right)^{*}ds}d\tau\\
              &  & +w_{*}(t_{0})e^{-\int_{t_{0}}^{t}\left(2\frac{\partial}{\partial s}\ln\left(\frac{1}{s}+\frac{\partial f}{\partial s}\right)+3\left(\frac{1}{s}+\frac{\partial f}{\partial s}\right)\right)^{*}ds},\end{eqnarray*}
and
\begin{eqnarray*}
                        u^{-2}\left(t,x\right) & \leq & \int_{t_{0}}^{t}\frac{1}{2\tau^{2}}\left(\frac{R_{f}-\tau^{2}\bar{R}}{\frac{1}{\tau}
                        +\frac{\partial f}{\partial\tau}}\right)^{*}e^{-\int_{\tau}^{t}\left(2\frac{\partial}{\partial s}\ln\left(\frac{1}{s}+\frac{\partial f}{\partial s}\right)+3\left(\frac{1}{s}+\frac{\partial f}{\partial s}\right)\right)_{*}ds}d\tau\\
              &  & +w^{*}(t_{0})e^{-\int_{t_{0}}^{t}\left(2\frac{\partial}{\partial s}\ln\left(\frac{1}{s}+\frac{\partial f}{\partial s}\right)+3\left(\frac{1}{s}+\frac{\partial f}{\partial s}\right)\right)_{*}ds}.
     \end{eqnarray*}
If we further assume that $\bar{R}$ is defined on $N$
such that the functions
      \begin{eqnarray*}
               \delta_{*}(t) 
                            & = &
                             \int_{1}^{t}\frac{1}{2\tau^{2}}\left(\frac{R_{f}-\tau^{2}\bar{R}}{\frac{1}{\tau}+\frac{\partial f}{\partial\tau}}\right)_{*}
                             e^{-\int_{\tau}^{t}\left(2\frac{\partial}{\partial s}\ln\left(\frac{1}{s}+\frac{\partial f}{\partial s}\right)+3\left(\frac{1}{s}+\frac{\partial f}{\partial s}\right)\right)^{*}ds}d\tau,
 \end{eqnarray*}
and                            
\begin{eqnarray*}                             
             \delta^{*}(t) 
                            & = &
                              \int_{1}^{t}\frac{1}{2\tau^{2}}\left(\frac{R_{f}-\tau^{2}\bar{R}}{\frac{1}{\tau}+\frac{\partial f}{\partial\tau}}\right)^{*}
                              e^{-\int_{\tau}^{t}\left(2\frac{\partial}{\partial s}\ln\left(\frac{1}{s}+\frac{\partial f}{\partial s}\right)+3\left(\frac{1}{s}+\frac{\partial f}{\partial s}\right)\right)_{*}ds}d\tau
       \end{eqnarray*}
are defined and finite for all $t\in[1,\infty)$, then the estimates
may be rewritten as 
      \begin{eqnarray*}
               u^{-2}(t,x) 
                         & \geq & 
                                    \delta_{*}(t)+\left(u^{*-2}(t_{0})-\delta_{*}(t_{0})\right)e^{-\int_{t_{0}}^{t}\left(2\frac{\partial}{\partial s}\ln\left(\frac{1}{s}+\frac{\partial f}{\partial s}\right)+3\left(\frac{1}{s}+\frac{\partial f}{\partial s}\right)\right)^{*}ds},
\end{eqnarray*}
and                            
\begin{eqnarray*}      
              u^{-2}(t,x) 
                         & \leq & 
                                   \delta^{*}(t)+\left(u_{*}^{-2}(t_{0})-\delta^{*}(t_{0})\right)e^{-\int_{t_{0}}^{t}\left(2\frac{\partial}{\partial s}\ln\left(\frac{1}{s}+\frac{\partial f}{\partial s}\right)+3\left(\frac{1}{s}+\frac{\partial f}{\partial s}\right)\right)_{*}ds}.
       \end{eqnarray*}
 \end{prop}

%%%% Proof of Proposition 5

\begin{proof}
Applying the parabolic maximum principle to (\ref{eq:w}) for $w=u^{-2}$ gives 
       \begin{eqnarray*}
               w_{*}'(t) 
                      & \geq & 
                              -\left(2\frac{\partial}{\partial t}\ln\left(\frac{1}{t}+\frac{\partial f}{\partial t}\right)+3\left(\frac{1}{t}+\frac{\partial f}{\partial t}\right)\right)^{*}w_{*}\\
                      &  & +\frac{1}{2t^{2}}\left(\left(R_{f}-t^{2}\bar{R}\right)\left(\frac{1}{t}+\frac{\partial f}{\partial t}\right)^{-1}\right)_{*}.
         \end{eqnarray*}
Solving the associated O.D.E., we have 
     \begin{eqnarray*}
            w_{*}(t) 
                    & \geq &
                                   \int_{t_{0}}^{t}\frac{1}{2\tau^{2}}\left(\frac{R_{f}-\tau^{2}\bar{R}}{\frac{1}{\tau}+\frac{\partial f}{\partial\tau}}\right)_{*}
                                        e^{\int_{t_{0}}^{\tau}\left(2\frac{\partial}{\partial s}\ln\left(\frac{1}{s}+\frac{\partial f}{\partial s}\right)+3\left(\frac{1}{s}+\frac{\partial f}{\partial s}\right)\right)^{*}ds}d\tau \\
                                 && \times e^{-\int_{t_{0}}^{t}\left(2\frac{\partial}{\partial s}\ln\left(\frac{1}{s}+\frac{\partial f}{\partial s}\right)+3\left(\frac{1}{s}+\frac{\partial f}{\partial s}\right)\right)^{*}ds}\\
                                 && +w_*(t_{0})e^{-\int_{t_{0}}^{t}\left(2\frac{\partial}{\partial s}\ln\left(\frac{1}{s}+\frac{\partial f}{\partial s}\right)+3\left(\frac{1}{s}+\frac{\partial f}{\partial s}\right)\right)^{*}ds}\\
                    & = & 
                                   \int_{t_{0}}^{t}\frac{1}{2\tau^{2}}\left(\frac{R_{f}-\tau^{2}\bar{R}}{\frac{1}{\tau}+\frac{\partial f}{\partial\tau}}\right)_{*}
                                         e^{-\int_{\tau}^{t}\left(2\frac{\partial}{\partial s}\ln\left(\frac{1}{s}+\frac{\partial f}{\partial s}\right)+3\left(\frac{1}{s}+\frac{\partial f}{\partial s}\right)\right)^{*}ds}d\tau\\
                               &  & +w_{*}(t_{0})e^{-\int_{t_{0}}^{t}\left(2\frac{\partial}{\partial s}\ln\left(\frac{1}{s}+\frac{\partial f}{\partial s}\right)+3\left(\frac{1}{s}+\frac{\partial f}{\partial s}\right)\right)^{*}ds}.
       \end{eqnarray*}
Therefore
        \[
              w(t)\geq\delta_{*}(t)+\left(w_*(t_{0})-\delta_{*}(t_{0})\right)e^{-\int_{t_{0}}^{t}\left(2\frac{\partial}{\partial s}\ln\left(\frac{1}{s}+\frac{\partial f}{\partial s}\right)+3\left(\frac{1}{s}+\frac{\partial f}{\partial s}\right)\right)^{*}ds}.
       \]
Similarly, applying the maximum principle to $w^{*}$, we get the
upper bound for $u^{-2}$.
\end{proof}
%%%% Proof 5 ends

%%% Proposition 6
\begin{prop}
\label{prop:interior est}
Let $I=[1,t_{1}]$ and $I'=[t_{0},t_{1}]$
with $1<t_{0}<t_{1}$, and suppose $u\in C^{2+\alpha}(A_{I})$ is
a solution to (\ref{eq:u}) on $A_{I}$ with source functions $\bar{R}$
and $f$ such that 
\[
0<f_{0}\leq\frac{1}{t}+\frac{\partial f}{\partial t}\leq f_{0}^{-1}\quad\mbox{for all }(t,x)\in A_{I},
\]
for some constant $f_{0}>0.$ Further suppose there is a constant
$\delta_{0}>0$ such that 
\[
0<\delta_{0}\leq u^{-2}(x,t)\leq\delta_{0}^{-1}\quad\mbox{for all }(t,x)\in A_{I}.
\]
Then 
        \begin{equation}
               \|u\|_{2+\alpha,I'}
                       \leq C
                                \left(\left\| \frac{\partial f}{\partial t}\right\| _{\alpha,I},\left\| \frac{\partial^{2}f}{\partial t^{2}}\right\| _{\alpha,I},\left\| R_{f}\right\| _{\alpha,I},\left\| \bar{R}\right\| _{\alpha,I}\right),\label{eq:est u}
         \end{equation}
where $C$ is a constant dependent on $t_{0},t_{1},f_{0},\delta_{0},\left\| \nabla f\right\| _{0,I}$.

With ${\displaystyle m=\frac{t}{2}\left(1-u^{-2}\right)}$, there
is a constant $C$ such that 
         \begin{equation}
                \left\| m\right\| _{2+\alpha,I'}
                         \leq C
                               \left(\left\| \frac{\partial f}{\partial t}\right\| _{\alpha,I},\left\| \frac{\partial^{2}f}{\partial t^{2}}\right\| _{\alpha,I},\left\| 1-\frac{1}{2}R_{f}\right\| _{\alpha,I},\left\| \bar{R}\right\| _{\alpha,I}\right),\label{est m}
         \end{equation}
where $C$ depends on $t_{0},t_{1},f_{0},\delta_{0},\left\| \nabla f\right\| _{0,I}$.
\end{prop}

%%%% Proof of Proposition 6

\begin{proof}
Let $s=\ln t.$ Then ${\displaystyle \frac{\partial}{\partial s}=t\frac{\partial}{\partial t}}$.
Let ${\displaystyle \gamma=\frac{1}{2}\left(1+t\frac{\partial f}{\partial t}\right)^{-1}}.$
        \begin{eqnarray*}
                    \frac{\partial u}{\partial s} 
                            & = & 
                                     \gamma u^{2}\Delta_{f}u+t^{2}\gamma\left(2\frac{\partial}{\partial t}\left(\frac{1}{t}+\frac{\partial f}{\partial t}\right)+3\left(\frac{1}{t}+\frac{\partial f}{\partial t}\right)^{2}\right)u\\
                                     &&-\frac{1}{2}\gamma\left(R_{f}-t^{2}\bar{R}\right)u^{3}\\
                            & = & \frac{\partial}{\partial x^{i}}\left(\gamma u^{2}e^{-2f} \sigma^{ij}\frac{\partial u}{\partial x^{j}}\right)-\frac{\partial}{\partial x^{i}}
                                           \left(\frac{\gamma}{e^{f}\sqrt{\det \sigma}}\right)e^{-f}\sqrt{\det \sigma}\sigma^{ij}\frac{\partial u}{\partial x^{j}}u^{2}\\
                                    &  & -2\gamma e^{-2f}\sigma^{ij}\frac{\partial u}{\partial x^{i}}\frac{\partial u}{\partial x^{j}}u+t^{2}\gamma\left(2\frac{\partial}{\partial t}
                                            \left(\frac{1}{t}+\frac{\partial f}{\partial t}\right)+3\left(\frac{1}{t}+\frac{\partial f}{\partial t}\right)^{2}\right)u\\
                                    &  & -\frac{1}{2}\gamma\left(R_{f}-t^{2}\bar{R}\right)u^{3}\\
                            & = & \frac{\partial}{\partial x^{i}}a^{i}(x,t,u,\partial u)-a(x,t,u,\partial u),
        \end{eqnarray*}
where $a(x,t,u,p)$ and $a^{i}(x,t,u,p)$ are functions defined by
       \[
             {\displaystyle a^{i}(x,t,u,p)=\gamma u^{2}e^{-2f}\sigma^{ij}p_j}
        \]
 and
        \begin{eqnarray*}
                    a(x,t,u,p) 
                          & = & \frac{\partial}{\partial x^{i}}\left(\frac{\gamma}{e^{f}\sqrt{\det \sigma}}\right)e^{-f}\sqrt{\det \sigma}\sigma^{ij}p_{j}u^{2}+2\gamma e^{-2f}\sigma^{ij}p_{i}p_{j}u\\
                           &  & -t^{2}\gamma\left(2\frac{\partial}{\partial t}\left(\frac{1}{t}+\frac{\partial f}{\partial t}\right)+3\left(\frac{1}{t}+\frac{\partial f}{\partial t}\right)^{2}\right)u+\frac{1}{2}\gamma\left(R_{f}-t^{2}\bar{R}\right)u^{3}
        \end{eqnarray*}
respectively. By the assumption, the functions $a$ and $a^{i}$ satisfy
that 
       \begin{eqnarray*}
                  a^{i}p_{i} &\geq& C|p|^{2},  \\
                       |a^{i}| &\leq& C'|p|, \\
       \end{eqnarray*}
and
       $$
               |a| \leq C''(1+|p|^{2}),
       $$
for some positive constants $C,$ $C',$ and $C''.$ Applying \cite[Theorem V.1.1]{LSU68},
we have
         \[
                ||u||_{\alpha',I''}\leq C_{1},
         \]
for some $0<\alpha'<1,\alpha'=\alpha(\delta_{0})$ and $C_{1}$ dependent
on $t_{0}$, $t_{1}$, $f_{0}$, $\delta_{0}$, $\left\| \frac{\partial f}{\partial t}\right\| _{0,I}$, $\left\| \frac{\partial^{2}f}{\partial t^{2}}\right\| _{0,I}$
$||\nabla f||_{0,I,}||R_{f}||_{0,I,}||\bar{R}||_{0,I}$ where $I'\subset I''\subset I$. 

Without loss of generality, we may assume $\alpha'\leq\alpha.$ The
usual Schauder interior estimates \cite[Theorem IV. 10.1]{LSU68}
give
           \[
                  ||u||_{2+\alpha',I'}\leq C_{2}\left(C_{1,}\left\| \frac{\partial f}{\partial t}\right\| _{\alpha,I},\left\| \frac{\partial^{2}f}{\partial t^{2}}\right\| _{\alpha,I},\left\| R_{f}\right\| _{\alpha,I},\left\| \bar{R}\right\| _{\alpha,I}\right).
           \]
Applying the Schauder estimates again to (\ref{eq:m}), we get
          \begin{eqnarray*}
                   \left(\frac{1}{t}+\frac{\partial f}{\partial t}\right)\frac{\partial m}{\partial t} 
                             & = & \frac{u^2}{2t^2}\Delta_{f}m+\frac{3u}{2t^{2}}\nabla u\cdot\nabla m-\left(2\frac{\partial^{2}f}{\partial t^{2}}+\frac{5}{t}\frac{\partial f}{\partial t}+3\left(\frac{\partial f}{\partial t}\right)^{2}\right)m\nonumber \\
                               &  & -\frac{1}{4t}\left(R_{f}-2-t^{2}\bar{R}-4t^{2}\frac{\partial^{2}f}{\partial t^{2}}-12t\frac{\partial f}{\partial t}-6\left(t\frac{\partial f}{\partial t}\right)^{2}\right).
           \end{eqnarray*}
The desired estimate (\ref{est m}) follows. 
\end{proof}
%%%% Proof 6 ends

With the a priori estimates from above, we are ready to prove the
existence of the solution of the initial value problem (\ref{eq:u})
and (\ref{eq:initial}).

%%%%Theorem 7

\begin{thm} \label{global existence}
\label{thm:existence} Assume $\bar{R}\in C^{\alpha}(A_{[t_{0},\infty)})$
and $f\in C^{2+\alpha}(A_{[t_{0},\infty)})$, $t_{0}>1$ such that the function
satisfies 
           \begin{equation}
                        0<1+t\frac{\partial f}{\partial t}<\infty\mbox{ for all }t_{0}\leq t<\infty.\label{parabolicity}
           \end{equation}
Further assume the nonnegative constant $K$ defined by
             \[
                         K=\sup_{t_{0}\leq t<\infty}\left\{ -\int_{t_{0}}^{t}\frac{1}{2\tau^{2}}\left(\frac{R_{f}-\tau^{2}\bar{R}}{\frac{1}{\tau}+\frac{\partial f}{\partial\tau}}\right)_{*}
                                e^{\int_{t_{0}}^{\tau}\left(2\frac{\partial}{\partial s}\ln\left(\frac{1}{s}+\frac{\partial f}{\partial s}\right)+3\left(\frac{1}{s}+\frac{\partial f}{\partial s}\right)\right)^{*}ds}d\tau\right\} 
              \]
satisfies $K<\infty$. Then for every $\varphi(x)\in C^{2,\alpha}\left(\Sigma\right)$
such that 
               \[
                          0<\varphi(x)<\frac{1}{\sqrt{K}}\mbox{ for all }x\in\Sigma,
                \]
there is a unique positive solution $u\in C^{2+\alpha}\left(A_{[t_{0},\infty)}\right)$
to (\ref{eq:u}) with initial condition
                \[
                       u(t_{0},\cdot)=\varphi(\cdot).
                \]
\end{thm}

%%%% Proof of Theorem 7

\begin{proof}
{}Define $\tilde{u}(t,x)=u(t_{0}t,x)$. Observe that $u\in C^{2+\alpha}(A_{[t_{0},\infty)})$
satisfies (\ref{eq:u}) if and only if $\tilde{u}$ on the interval
$[1,\infty)$ satisfies
                   \begin{eqnarray}
                                     \left(\frac{1}{t}+\frac{\partial\tilde{f}}{\partial t}\right)\frac{\partial\tilde{u}}{\partial t} 
                                                        & = & \frac{1}{2t^{2}}\tilde{u}^{2}\Delta_{f}\tilde{u}+\left[\frac{\partial}{\partial t}\left(\frac{1}{t}+\frac{\partial\tilde{f}}{\partial t}\right)+\frac{3}{2}\left(\frac{1}{t}+\frac{\partial\tilde{f}}{\partial t}\right)^{2}\right]\tilde{u}\nonumber \\
                                  &  & -\frac{1}{4t^2}\left(\tilde{R}_f-t^2\tilde{R}\right)\tilde{u}^{3},\label{eq:tilde u}
                   \end{eqnarray}
where 
                   \[
                                  \tilde{f}(t,x)=f(t_{0}t,x),\tilde{\, R}_{f}(t,x)=R_{f}(t_{0}t,x),\,\hbox{and}\,\tilde{R}(t,x)=t_0^{2}\bar{R}(t_{0}t,x).
                   \]
Denoting the estimating functions of Proposition \ref{prop: C0bounds}
by $\tilde{\delta}^{*}(t)$ and $\tilde{\delta}_{*}(t)$, we have that 
                   \[
                                   \tilde{\delta}^{*}(t)=\delta^{*}(t_{0}t),\quad\tilde{\delta}_{*}(t)=\delta_{*}(t_{0}t)\mbox{ for all }1\leq t<\infty.
                    \]
It can be verified that 
                    \[
                            K=\sup_{1\leq t<\infty}\left\{ -\int_{1}^{t}\frac{1}{2\tau^{2}}\left(\frac{\tilde{R}_{f}-\tau^{2}\tilde{R}}{\frac{1}{\tau}+\frac{\partial\tilde{f}}{\partial\tau}}\right)_{*}
                                     e^{\int_{1}^{\tau}\left(2\frac{\partial}{\partial s}\ln\left(\frac{1}{s}+\frac{\partial\tilde{f}}{\partial s}\right)+3\left(\frac{1}{s}+\frac{\partial\tilde{f}}{\partial s}\right)\right)^{*}ds}d\tau\right\} .
                    \]
The upper bound 
                    \[
                             \varphi(x)<\frac{1}{\sqrt{K}}\quad\hbox{for\, all\,}\, x\in\Sigma
                    \]
implies
                    \begin{equation}
                                  \tilde{\delta}_{*}(t)+\tilde{u}^{*-2}(1)\exp\left(\int_{1}^{t}
                                                 \left(2\frac{\partial}{\partial s}\ln\left(\frac{1}{s}+\frac{\partial\tilde{f}}{\partial s}\right)+3\left(\frac{1}{s}+\frac{\partial\tilde{f}}{\partial s}\right)\right)^{*}ds\right)
                                   > 0\label{eq:pos}              
                     \end{equation}
for all $t\geq1$. Since the foliation satisfies (\ref{parabolicity}),
the equation (\ref{eq:tilde u}) for $\tilde{u}$ is parabolic. The
short-time existence of solutions can be obtained from Schauder theory
and the implicit function theorem. By Propositions \ref{prop: C0bounds}
and \ref{prop:interior est}, there is $T>0$ and $\tilde{u}\in C^{2+\alpha}(A_{[1,1+T]})$
satisfying (\ref{eq:tilde u}) and the initial condition $\tilde{u}(1,x)=\varphi(x)$
on $[1,1+T]$. Moreover by Proposition \ref{prop: C0bounds} and (\ref{eq:pos}),
there are functions $0<\delta_{1}(t)\leq\delta_{2}(t)<\infty,$ $t\geq1,$
independent of $T$, such that
                  \[
                            0 < \delta_{1}(t)\leq\tilde{u}^{-2}(t,x)\leq\delta_{2}(t) < \infty\quad\mbox{ for }\,1\leq t\leq1+T.
                  \]
Let $U=\left\{ t\in\mathbb{R}^{+}:\,\exists\,\tilde{u}\in C^{2+\alpha}(A_{[1,1+t]})\mbox{ satisfying (\ref{eq:tilde u}) and }\tilde{u}(1,x)=\varphi(x)\right\} .$
The local existence guarantees that $U$ is open in $\mathbb{{R}^{+}}.$
Since $[1,1+t]$ is compact, there is a constant $\delta_{0}$ such
that $0<\delta_{0}\leq\tilde{u}^{-2}(t,x)\leq\delta_{0}^{-1}$ for
all $(t,x)\in A_{[1,1+t]}.$ From the interior estimates (\ref{eq:est u})
of Proposition \ref{prop:interior est}, we have an a priori estimate
for $||\tilde{u}(1+t,\cdot)||_{2,\alpha}$. By the local existence,
the solution can be extended to $A_{[1,1+t+T]}$ for some $T$ independent
of $\tilde{u},$ which shows that $U$ is closed. Hence $\tilde{u}$
extends to a global solution $\tilde{u}\in C^{2+\alpha}(A_{[1,\infty)})$
and the function $u(t,x)=\tilde{u}(t/t_{0},x)$ is the required solution.
\end{proof}
%%%% Proof 7 ends.

Theorem \ref{thm:existence} gives the existence of the initial value
problem (\ref{eq:u}) and (\ref{eq:initial}); however, the asymptotic
behavior of the solution is still not controlled yet. In Theorem \ref{thm:AF},
we will describe decay conditions on the source functions $\bar{R}$
and $f(t,x)$ which ensure existence of solutions satisfying the boundary
behavior. 

From now on we use $C$ to denote a constant, but it may vary from
line to line.

%%%% Lemma 8
\begin{lem}
\label{lem: af ineq} 
Assume that $f$ satisfies 
              \[
                      0<1+t\frac{\partial f}{\partial t}<\infty\mbox{ for all }1\leq t<\infty.
              \]
Suppose 
              \begin{equation}
                          t\left(\frac{\partial f}{\partial t}\right)^{*}\in L^{1}\left([1,\infty)\right)\label{L1 bound}
               \end{equation}
and 
               \begin{equation}
                         t\left(\frac{\partial}{\partial t}\ln\left(\frac{1}{t}+\frac{\partial f}{\partial t}\right)\right)^{*}-t\frac{d}{dt}\ln\left(\frac{1}{t}+\frac{\partial f}{\partial t}\right)^{*}\in L^{1}\left([1,\infty)\right).\label{lemma 8}
                \end{equation}
Then there is a constant $C$ such that 
               \begin{equation}
                            1-\frac{C}{t}
                                     \leq
                                            \int_{1}^{t}\frac{1}{\tau^{2}}\left(\frac{1}{\tau}+\frac{\partial f}{\partial\tau}\right)^{*-1}
                                                           e^{-\int_{\tau}^{t}\left(2\frac{\partial}{\partial s}\ln\left(\frac{1}{s}+\frac{\partial f}{\partial s}\right)+3\left(\frac{1}{s}+\frac{\partial f}{\partial s}\right)\right)^{*}ds}d\tau\leq1+\frac{C}{t},\label{AF-lower}
                \end{equation}
and 
               \begin{equation}
                              1-\frac{C}{t}
                                       \leq
                                              \int_{1}^{t}\frac{1}{\tau^{2}}\left(\frac{1}{\tau}+\frac{\partial f}{\partial\tau}\right)_{*}^{-1}
                                                             e^{-\int_{\tau}^{t}\left(2\frac{\partial}{\partial s}\ln\left(\frac{1}{s}+\frac{\partial f}{\partial s}\right)+3\left(\frac{1}{s}+\frac{\partial f}{\partial s}\right)\right)_{*}ds}d\tau\leq1+\frac{C}{t},\label{AF-upper}
                \end{equation}
for all $t\geq1$. 
\end{lem}

%%%%Proof of Lemma 8

\begin{proof}
To show (\ref{AF-lower}), it suffices to show that 
                   \begin{eqnarray*}
                                &  &   \left|
                                                    \int_{1}^{t}  \frac{t}{\tau^{2}}\left(\frac{1}{\tau}+\frac{\partial f}{\partial\tau}\right)^{*-1}
                                                     e^{-\int_{\tau}^{t}\left(2\frac{\partial}{\partial s}\ln\left(\frac{1}{s}+\frac{\partial f}{\partial s}\right)+3\left(\frac{1}{s}+\frac{\partial f}{\partial s}\right)\right)^{*}ds}d\tau-t
                                           \right|\\
                                & \leq &
                                            \left|
                                                   \int_{1}^{t}\frac{t}{\tau^{2}}\left(\frac{1}{\tau}+\frac{\partial f}{\partial\tau}\right)^{*-1}
                                                     e^{-\int_{\tau}^{t}3\left(\frac{1}{s}+\frac{\partial f}{\partial s}\right)^{*}ds}e^{-\int_{\tau}^{t}2\frac{d}{ds}\ln\left(\frac{1}{s}+\frac{\partial f}{\partial s}\right)^{*}ds}d\tau-t
                                           \right|\\
                                          &  &
                                               +\bigg|
                                                       \int_{1}^{t}\frac{t}{\tau^{2}}\left(\frac{1}{\tau}+\frac{\partial f}{\partial\tau}\right)^{*-1}e^{-\int_{\tau}^{t}3\left(\frac{1}{s}+\frac{\partial f}{\partial s}\right)^{*}ds} \\
                                                 && \qquad \times \left(e^{-\int_{\tau}^{t}\left(2\frac{\partial}{\partial s}\ln\left(\frac{1}{s}+\frac{\partial f}{\partial s}\right)\right)^{*}ds}-e^{-\int_{\tau}^{t}2\frac{d}{ds}\ln\left(\frac{1}{s}+\frac{\partial f}{\partial s}\right)^{*}ds}\right)d\tau
                                                 \bigg|\\
                                & = & I+II\\
                                & < & C.
                      \end{eqnarray*}
From (\ref{L1 bound}), there is $t_{0}>1$ such that $\int_{t_{0}}^{\infty}3t\left(\frac{\partial f}{\partial t}\right)^{*}dt<1.$
Using 
\[
                                  \left|e^{\eta}-1\right|\leq2\left|\eta\right|\mbox{ for }\left|\eta\right|\leq1,
 \]
we see
\begin{eqnarray*}
                             I & = &
                                           \left|
                                                          \int_{1}^{t}\frac{t}{\tau^{2}}\left(\frac{1}{\tau}+\frac{\partial f}{\partial\tau}\right)^{*-1}
                                                                   e^{-\int_{\tau}^{t}3\left(\frac{1}{s}+\frac{\partial f}{\partial s}\right)^{*}ds}e^{-\int_{\tau}^{t}2\frac{d}{ds}\ln\left(\frac{1}{s}+\frac{\partial f}{\partial s}\right)^{*}ds}d\tau-t 
                                            \right|\\  
                                       & =&    \left| \int_1^t  
                                       \frac{ \left(1+\tau \frac{\partial f}{\partial \tau}\right)^*}{ \left( 1+ t \frac{\partial f}{\partial t}\right)^{*2}} e^{-\int_{\tau}^{t}3\left(\frac{\partial f}{\partial s}\right)^{*}ds}d\tau-t \right|\\  
                                    % & = & 
                                       %        \left|
                                          %                \int_{1}^{t_{0}}\frac{t}{\tau^{2}}\left(\frac{1}{\tau}+\frac{\partial f}{\partial\tau}\right)^{*-1}
                                             %                         e^{-\int_{\tau}^{t}3\left(\frac{1}{s}+\frac{\partial f}{\partial s}\right)^{*}ds}e^{-\int_{\tau}^{t}2\frac{d}{ds}\ln\left(\frac{1}{s}+\frac{\partial f}{\partial s}\right)^{*}ds}d\tau+t_{0}
                                              %\right|\\
                                    % &  & 
                                    % +\left|
                                       %                    \int_{t_{0}}^{t}\frac{t}{\tau^{2}}\left(\frac{1}{\tau}+\frac{\partial f}{\partial\tau}\right)^{*-1}
                                          %                               e^{-\int_{\tau}^{t}3\left(\frac{1}{s}+\frac{\partial f}{\partial s}\right)^{*}ds}e^{-\int_{\tau}^{t}2\frac{d}{ds}\ln\left(\frac{1}{s}+\frac{\partial f}{\partial s}\right)^{*}ds}-1d\tau
                                             %   \right|\\
                                   & \leq & C+\int_{t_{0}}^{t}\frac{\left(1+\tau\frac{\partial f}{\partial\tau}\right)^{*}}{\left(1+t\frac{\partial f}{\partial t}\right)^{*2}}\left|e^{-\int_{\tau}^{t}3\left(\frac{\partial f}{\partial s}\right)^{*}ds}-1\right|d\tau 
                                   +\int_{t_0}^t 
                                   \left|
                                         \frac{\left(1+\tau\frac{\partial f}{\partial\tau}\right)^{*}}{\left(1+t\frac{\partial f}{\partial t}\right)^{*2}} -1 \right| d\tau
                                   \\
                                   & \leq & C+C\int_{t_{0}}^{t}\int_{\tau}^{t}\left|\frac{\partial f}{\partial s}\right|^{*}dsd\tau \\
                                   &&
                                   +  \frac{1}{\left(1+t \frac{\partial f}{\partial t}\right)^{*2}}  
                                   \int_{t_0}^t 
                                           \left|
                                                \left( \tau \frac{\partial f }{\partial \tau}  \right)^*-2t\left( \frac{\partial f}{\partial t} \right)^*
                                                + t^2 \left( \frac{\partial f}{\partial t} \right)^{*2}
                                           \right|  d\tau
                                   \\
                                   & = & C+C\int_{t_{0}}^{t}\left|\frac{\partial f}{\partial s}\right|^{*}\left(s-t_{0}\right)ds<C.
                             \end{eqnarray*}
Similarly, from (\ref{lemma 8}), there is $t_{0}>1$ such that 
                    \[
                                    \int_{t_0}^{\infty}  
                                    \left|
                                            \left(  
                                                       2\frac{\partial}{\partial s} \ln
                                                               \left(
                                                                         \frac{1}{s}+\frac{\partial f}{\partial s}
                                                               \right)
                                             \right)^*
                                                       -2\frac{d}{ds}\ln
                                                               \left(
                                                                          \frac{1}{s}+\frac{\partial f}{\partial s}
                                                               \right)^*
                                     \right|
                                     ds<1.
                     \]
We see 

                   \begin{eqnarray*}
                  II &=&
                  \bigg|
                                                       \int_{1}^{t}\frac{t}{\tau^{2}}\left(\frac{1}{\tau}+\frac{\partial f}{\partial\tau}\right)^{*-1}e^{-\int_{\tau}^{t}3\left(\frac{1}{s}+\frac{\partial f}{\partial s}\right)^{*}ds} \\
                                                 && \qquad \times \left(e^{-\int_{\tau}^{t}\left(2\frac{\partial}{\partial s}\ln\left(\frac{1}{s}+\frac{\partial f}{\partial s}\right)\right)^{*}ds}-e^{-\int_{\tau}^{t}2\frac{d}{ds}\ln\left(\frac{1}{s}+\frac{\partial f}{\partial s}\right)^{*}ds}\right)d\tau
                                                 \bigg|\\
                       & \leq & 
                               C + C \int_{t_{0}}^{t}\left|e^{-\int_{\tau}^{t}
                                      \left(
                                           2\frac{\partial}{\partial s}\ln
                                                    \left(
                                                             \frac{1}{s}+\frac{\partial f}{\partial s}
                                                    \right)
                                        \right)^*
                                -2\frac{d}{ds}\ln\left(\frac{1}{s}+\frac{\partial f}{\partial s}\right)^{*}ds}-1\right|d\tau\\
                          & \leq & C+C\int_{t_{0}}^{t}\int_{\tau}^{t}
                                               \left(
                                                  \frac{\partial}{\partial s}\ln
                                                          \left(
                                                                 \frac{1}{s}+\frac{\partial f}{\partial s}
                                                          \right)
                                                \right)^*
                                          -\frac{d}{ds}\ln 
                                                  \left(
                                                                 \frac{1}{s}+\frac{\partial f}{\partial s}
                                                    \right)^{*}dsd\tau\\
                          & = & C+C\int_{t_{0}}^{t}
                                                  \left(
                                                         s-t_{0}
                                                   \right)                                                  
                                                   \left(
                                                            \frac{\partial}{\partial s}\ln
                                                                  \left( 
                                                                          \frac{1}{s}+\frac{\partial f}{\partial s}
                                                                   \right)
                                                   \right)^{*}
                                    -\frac{d}{ds}\ln
                                                   \left(
                                                                   \frac{1}{s}+\frac{\partial f}{\partial s}
                                                    \right)^{*}ds<C.
                        \end{eqnarray*}
\end{proof}
%%%% Proof 8 ends.

%%% Theorem 9
\begin{thm}
\label{thm:AF}Let $u\in C^{2+\alpha}(N)$ be a solution of (\ref{eq:u}).
Suppose that  $f\in C^{4+\alpha}(N)$ satisfies 
\begin{equation}
{\displaystyle \left(t\frac{\partial f}{\partial t}\right)^{*}\in L^{1}\left([1,\infty)\right)\,} \label{eq:f}
\end{equation}
and
\begin{equation}  \label{thm 9-lemma 8}
t\left(\frac{\partial}{\partial t}\ln\left(\frac{1}{t}+\frac{\partial f}{\partial t}\right)\right)^{*}-t\frac{d}{dt}\ln\left(\frac{1}{t}+\frac{\partial f}{\partial t}\right)^{*}\in L^{1}\left([1,\infty)\right).
\end{equation}
Further assume  there is a constant $C>0$ such that for all $t\geq2$ and $I_{t}=[t/2,2t]$, 

\begin{equation} \label{eq:decay}
{\displaystyle \left\| \bar{R}t^{2}\right\| _{\alpha,I_{t}}+  \left\| t\frac{\partial f}{\partial t}\right\| _{2+\alpha,I_{t}}  + \left\|1-e^{-2f}\right\|_{2+\alpha,I_{t}}  \leq\frac{C}{t} }
\end{equation}
and
\begin{equation}
\int_{1}^{\infty}|R_{f}-2|^{*}+|\bar{R}t^{2}|^{*}dt<\infty.\label{eq:af}
\end{equation}
Then $\bar{g}$ satisfies the asymptotically flat condition (\ref{AF condition}).
Moreover the Riemannian curvature $\bar{Rm}$ of the $ $3-metric $\bar{g}$
on $N$ is H\"{o}lder continuous and decays as ${\displaystyle \left|\bar{Rm}\right|<\frac{C}{t^{3}}}$,
and the ADM mass of $\left(N,\bar{g}\right)$ can be expressed as 
                \begin{eqnarray*}
                             m_{ADM} 
                                  &=&
                                        \frac{1}{4\pi}\lim_{t\rightarrow\infty}\oint_{S_{t}}\frac{t}{2}
                                                \left(
                                                        1-u^{-2}
                                                 \right)
                                        d\sigma.                  
                \end{eqnarray*}
\end{thm}

\begin{proof}
By Lemma \ref{lem: af ineq} and the decay conditions (\ref{eq:f}), (\ref{thm 9-lemma 8}), and (\ref{eq:af}), we have 
\begin{equation*}
1-\frac{C}{t}\leq\delta_{*}(t)\leq\delta^{*}(t)\leq1+\frac{C}{t}\label{af}
\end{equation*}
for all $t\geq t_{0,}$ for some $t_{0}$ large.

Define $\tilde{u}(t,x)=u(\tau t,x)$. Observe that $u\in C^{2+\alpha}(A_{[\tau/2,2\tau]})$
satisfies (\ref{eq:u}) if and only if $\tilde{u}$ on the interval
$[1/2,2]$ satisfies\begin{eqnarray*}
\left(\frac{1}{t}+\frac{\partial\tilde{f}}{\partial t}\right)\frac{\partial\tilde{u}}{\partial t} & = & \frac{1}{2t^{2}}\tilde{u}^{2}\Delta_{f}\tilde{u}+\left[\frac{\partial}{\partial t}\left(\frac{1}{t}+\frac{\partial\tilde{f}}{\partial t}\right)+\frac{3}{2}\left(\frac{1}{t}+\frac{\partial\tilde{f}}{\partial t}\right)^{2}\right]\tilde{u}\\
 &  & -\frac{1}{4t^{2}}\left(\tilde{R}_{f}-t^{2}\tilde{R}\right)\tilde{u}^{3},\end{eqnarray*}
where \[
\tilde{f}(t,x)=f(\tau t,x),\tilde{\, R}_{f}(t,x)=R_{f}(\tau t,x),\,\hbox{and}\,\tilde{R}(t,x)=\tau^{2}\bar{R}(\tau t,x).\]
Applying Proposition \ref{prop:interior est} to $\tilde{u}$
on the interval $[1/2,2]$ and then rescaling back, we obtain the estimates
$||u||_{2+\alpha,I_{t}}\leq C$. Also, $m$ scales as $\tilde{m}(t,x)=m(\tau t,x)/\tau$. The scaling argument and (\ref{est m}) yield
\begin{eqnarray}
||m||_{2+\alpha,I'_{t}} & \leq & Ct\left(\left\| \frac{\partial f}{\partial t}\right\| _{\alpha,I_t}+\left\| \frac{\partial^{2}f}{\partial t^{2}}\right\| _{\alpha,I_t}+\left\| 1-\frac{1}{2}R_{f}\right\| _{\alpha,I_t}+\left\| \tau^{2}\bar{R}\right\| _{\alpha,I_t}\right)\nonumber\\
 &  & +C\left\| m\right\| _{0,I_t} \label{rescaling m},
\end{eqnarray}
where $I'_{t}=[t,2t]$, $I_t=[t/2,2t]$, and $C$
is some constant independent of $u$ and $t.$ 
The bounds (\ref{af}) control $\left\| m\right\| _{0,I_t}$. Under conformal change, the scalar curvature $R_f$ satisfies
$$
 \frac{R_f}{ 2}  =e^{-2f}\left(1-\Delta_{\sigma} f \right).
$$
Thus the decay condition
(\ref{eq:decay}) controls the remaining terms, and 
\[
\left\| m\right\| _{2+\alpha,I'_{t}}\leq C
\]
is uniformly bounded for all $t\geq1.$ Expressing this in terms of
$u$ and derivatives of $u$ gives\[
\left\| 1-u^{-2}\right\| _{\alpha,I'_{t}}+\left\| t\partial_{t}u\right\| _{\alpha,I'_{t}}+\left\| \nabla u\right\| _{\alpha,I'_{t}}+\left\| \nabla^{2}u\right\| _{\alpha,I'_{t}}\leq\frac{C}{t},\]
which shows that $\bar{g}$ is asymptotically Euclidean. The estimate for $\nabla^2u$ shows that the Ricci curvature on $N$,  $\bar{Rc} \in C^{0,\alpha}(N)$ and $\left| \bar{Rc} \right| \leq C/t^3.$ Since $N$  is of dimension $3$, this controls the full curvature tensor.

From \cite{Bar86} the ADM mass is uniquely defined. We compute the
flux integral in terms of spherical coordinates. 
\begin{eqnarray*}
              m_{ADM}
                    & =& \frac{1}{16\pi}
                                   \oint_{S_{\infty}}
                                        \left(
                                                 \dot{g}^{ab}\dot{\nabla}_a\bar{g}_{bc} - \dot{\nabla}_{c}
                                                           \left(
                                                                 \dot{g}^{ab}\bar{g}_{ab}
                                                            \right)
                                         \right)
                                  d\sigma^c\\
                  &=& \frac{1}{16\pi} \lim_{t\rightarrow \infty}
                                    \oint_{S_t}    
                                                \left(
                                                         \dot{g}^{ab}\dot{\nabla}_a\bar{g}_{b3} - \dot{\nabla}_{3}
                                                                 \left(
                                                                        \dot{g}^{ab}\bar{g}_{ab}
                                                                  \right)
                                                 \right)u^{-2}t^2
                                    d\sigma,      
 \end{eqnarray*}
where $\dot{g}=dt^{2}+t^{2}\sigma$ is the flat metric, $\dot{\nabla}$
denotes covariant derivative of $\dot{g}$, and  $d\sigma$ is the area element on $S^{2}$. For $i,j=1,2$, the connections
of $\dot{g}$ are $\dot{\Gamma}_{33}^{3}=0$, ${\displaystyle \dot{\Gamma}_{ij}^{3}=-\frac{1}{t}\dot{g}_{ij}}$,
and ${\displaystyle \dot{\Gamma}_{3j}^{i}=\frac{1}{t}\delta_{j}^{i}}$.
Moreover\begin{eqnarray*}
\dot{g}^{ab}\dot{\nabla}_{a}\bar{g}_{b3} & = & \dot{g}^{33}\dot{\nabla}_{3}\bar{g}_{33}+\dot{g}^{ij}\dot{\nabla}_{i}\bar{g}_{j3}\\
 & = & \dot{g}^{33}\partial_{3}\bar{g}_{33}-\dot{g}^{ij}\dot{\Gamma}_{ij}^{3}\bar{g}_{33}-\dot{g}^{ij}\dot{\Gamma}_{i3}^{k}\bar{g}_{jk}\\
 & = & 2u\frac{\partial u}{\partial t}+\frac{2}{t}u^{2}-\frac{1}{t}\dot{g}^{jk}\bar{g}_{jk}\\
 & = & 2u\frac{\partial u}{\partial t}+\frac{2}{t}u^{2}-\frac{2}{t}e^{2f}\end{eqnarray*}
and\[
\dot{\nabla}_{3}\left(\dot{g}^{ab}\bar{g}_{ab}\right)=2u\frac{\partial u}{\partial t}+4e^{2f}\frac{\partial f}{\partial t}.\]
Hence 
 \begin{eqnarray*}
                      \left(
                                  \dot{g}^{ab}\dot{\nabla}_a\bar{g}_{b3} - \dot{\nabla}_{3}
                                                \left(
                                                         \dot{g}^{ab}\bar{g}_{ab}
                                                 \right)
                       \right)u^{-2}t^2 
                    %     & = &          \left(
                        %                             \frac{2}{t}u^{2}-\frac{2}{t}e^{2f}-4e^{2f}\frac{\partial   f}{\partial t}
                            %               \right)u^{-2}t^{2} \\
                          &=&     4   \left(
                                                              \frac{t}{2}-\frac{t}{2}e^{2f}u^{-2} - t^{2}e^{2f}\frac{\partial   f}{\partial t}  u^{-2}
                                                   \right) \\
                           &=&  4 \left( 
                                                 m -\frac{t}{2}  \left( e^{2f}-1 \right) u^{-2} - t^{2}e^{2f}\frac{\partial   f}{\partial t}  u^{-2}
                                         \right)
\end{eqnarray*}
From the decay condition  (\ref{eq:decay}), we see that $f\rightarrow 0$ as $t$ approaches $\infty$.
Since $t\frac{\partial f }{\partial t} \in L^1\left([1,\infty)\right)$ and $u^{-2}$ is bounded, we find that both
$\frac{t}{2}  \left( e^{2f}-1 \right) u^{-2}$ and $t^{2}e^{2f}\frac{\partial   f}{\partial t}  u^{-2}$ approach zero as $t$ approaches $\infty$. Now the mass integral reduces to 
\begin{eqnarray*}
      m_{ADM}   &=&   \frac{1}{4\pi}\lim_{t\rightarrow\infty}   \oint_{S_{t}} m   d\sigma.
\end{eqnarray*}

To show that the
limit exists, it suffices to show the following is integrable on $\left(t_{0},\infty\right)$,
\begin{eqnarray*}
   &&  \oint_{S_{t}}
                          \left(
                                          1+t\frac{\partial f}{\partial t}
                           \right)  
             \frac{\partial m}{\partial t}
                           d\sigma \\
         &=& \oint_{S_t}  \frac{u^2}{2t^2}\Delta_{f}m+\frac{3u}{2t^{2}}\nabla u\cdot\nabla m 
         -\left(2\frac{\partial^{2}f}{\partial t^{2}}+\frac{5}{t}\frac{\partial f}{\partial t}+3\left(\frac{\partial f}{\partial t}\right)^{2}\right)m\nonumber \\
 &  & -\frac{1}{4t}\left(R_{f}-2-t^{2}\bar{R}-4t^{2}\frac{\partial^{2}f}{\partial t^{2}}-12t\frac{\partial f}{\partial t}-6\left(t\frac{\partial f}{\partial t}\right)^{2} \right) d\sigma.
 \end{eqnarray*}
From decay conditions (\ref{eq:f}), (\ref{eq:decay}),
and (\ref{eq:af}), it can be shown that the right hand side of the
the equation is integrable on $\left(t_0,\infty\right)$ for some
large $t_0$ since $m$ and $\nabla m$ are bounded. 
\end{proof}

%%%%Proof 9 ends. 

Theorem \ref{thm:existence} shows the existence of the solution of
the initial value problem (\ref{eq:u}) and (\ref{eq:initial}). 
Theorem \ref{thm:AF} shows the asymptotical flatness of $\bar{g}=u^2 dt^2+t^2e^{2f}\sigma.$
They together give us Theorem \ref{thm:M1}. Next we prove Theorem
\ref{thm:M2}. Since the mean curvature 
$\displaystyle H=\frac{2}{u}
                 \left(
                            \frac{1}{t}+\frac{\partial f}{\partial t}
                 \right)$
and ${\displaystyle \frac{1}{t}+\frac{\partial f}{\partial t}>0},$
in order to have minimal boundary surface $\Sigma$, we need to solve
for $u$ so that the initial condition $u^{-1}(1,\cdot)=0$. 

%%%%%%
%Go back rewrite the statement of Theorem 1 !!!!!!!!!!!!!!!!
%%%%%%%

%%%% Theorem 2
\begin{thmM2}
Let $t\frac{\partial f}{\partial t}  \in C^{2+\alpha}\left(N\right)$ and $\bar{R} \in  C^{\alpha}\left(N\right)$ be given such that $\delta_*(t)$ and $\delta^*(t)$ defined in Proposition \ref{prop: C0bounds} are finite on $[1,\infty)$. Further suppose that for all $1\leq t<\infty$, 
\[
0<1+t\frac{\partial f}{\partial t} <\infty,
\]
and 
\[
t^{2}\bar{R}<R_{f}.
\]
Then there is $u^{-1}\in C^{2+\alpha}\left(A_{(1,\infty)}\right)$
such that the metric $\bar{g}$ on $N$ has curvature uniformly bounded
on $A_{[1,2]}$ with totally geodesic boundary. Let $0<\eta<1$ be
such that 
\begin{equation}
1-\eta<R_{f}-t^{2}\bar{R}|_{t=1}<\left(1-\eta\right)^{-1}.\label{bdy est}
\end{equation}
Then there is $t_{0}>1$ such that $1<t<t_{0},$
\[
\frac{t-1}{t}\left(1-\eta\right)<u^{-2}\left(t\right)<\frac{t-1}{t}\left(1-\eta\right)^{-1},
\]
which gives 
\[
1-\frac{\eta}{1-\eta}\left(t-1\right)\leq2m\leq1+\eta\left(t-1\right).
\]
\end{thmM2}

%%%% Proof of Theorem 2
\begin{proof}

We introduce the scaling transformation
\[
\tilde{u}\left(t\right)=\sqrt{\frac{t}{t+1}}u\left(t+1\right)\mbox{ where }t\in\left(0,\infty\right).
\]
The evolution equation for $\tilde{u}$ is 
\begin{eqnarray}
             \left(
                        \frac{1}{t}+\frac{\partial\tilde{f}}{\partial t}
             \right)
                      \frac{\partial\tilde{u}}{\partial t} 
                   & = & \frac{1}{2t^{2}}\tilde{u}^{2}\Delta_{f}\tilde{u}+
                            \left(
                                   \frac{\partial}{\partial t}\left(\frac{1}{t}+\frac{\partial\tilde{f}}{\partial t}\right)+\frac{3}{2}
                                           \left(
                                                   \frac{1}{t}+\frac{\partial\tilde{f}}{\partial t}
                                          \right)
                                           \left(
                                                   \frac{1}{t}+\frac{t}{t+1}\frac{\partial\tilde{f}}{\partial t}
                                            \right)
                               \right)\tilde{u}\label{eq:tilde{u}}\nonumber \\
 &  & -\frac{1}{4t^{2}}\left(\tilde{R}_{f}-t^{2}\tilde{R}\right)\tilde{u}^{3}
\end{eqnarray}
where 
\begin{eqnarray*}
\frac{\partial\tilde{f}}{\partial t}\left(t\right) & = & \frac{t+1}{t}\frac{\partial f}{\partial t}\left(t+1\right),\\
\tilde{R}_{f}\left(t\right) & = & R_{f}\left(t+1\right),\\
t^{2}\tilde{R}\left(t\right) & = & \left(t+1\right)^{2}\bar{R}\left(t+1\right).
\end{eqnarray*}
 Observe that the $\tilde{u}$-equation has a similar form as the $u$-equation (\ref{eq:u}). 
 %The existence Theorem \ref{global existence} uses only the fields $\bar{R}$,$f$, and $1+t\frac{\partial f}{\partial t}$. 
 The assumptions $0<1+t\frac{\partial f}{\partial t} <\infty$ and  $t^2\bar{R}<R_f$ imply that  
\begin{eqnarray*}
                            K
                                &=&\sup_{0< t<\infty}\left\{ -\int_{0}^{t}\frac{1}{2\tau^{2}}\left(\frac{\tilde{R}_{f}-\tau^{2}\tilde{R}}{\frac{1}{\tau}+\frac{\partial\tilde{f}}{\partial\tau}}\right)_{*}
                                     e^{\int_{1}^{\tau}\left(2\frac{\partial}{\partial s}\ln\left(\frac{1}{s}+\frac{\partial\tilde{f}}{\partial s}\right)+3\left(\frac{1}{s}+ \frac{s}{s+1}   \frac{\partial\tilde{f}}{\partial s}\right)\right)^{*}ds}d\tau\right\} \\
                                     &=&0.
\end{eqnarray*}

By the existence theorem, it can be shown that for any positive initial condition $\tilde{u}(t_0,\cdot)=\varphi(\cdot), t_0>0$, the solution $u^{-2}(t,x)>0$ satisfying $\tilde{u}(t_0,\cdot)=\varphi(\cdot)$ exists on $[t_0,\infty)$. Let $\varphi_{\epsilon} \in C^{2,\alpha}(\Sigma),0<\epsilon<1$, be a family of functions satisfying 
$$ 
      \tilde{\delta}_*(\epsilon) \leq \varphi_{\epsilon}^{-2} \leq \tilde{\delta}^*(\epsilon),
$$      
where 
\[
\tilde{\delta}_{*}\left(t\right)=\int_{0}^{t}\frac{1}{2\tau^{2}}\left(\frac{\tilde{R}_{f}-\tau^{2}\tilde{R}}{\frac{1}{\tau}+\frac{\partial\tilde{f}}{\partial t}}\right)_{*}e^{-\int_{\tau}^{t}\left(2\frac{\partial}{\partial s}\ln\left(\frac{1}{s}+\frac{\partial\tilde{f}}{\partial s}\right)+3\left(\frac{1}{s}+\frac{s}{s+1}\frac{\partial\tilde{f}}{\partial s}\right)\right)^{*}ds}d\tau
\]
and 
\[
\tilde{\delta}^{*}\left(t\right)=\int_{0}^{t}\frac{1}{2\tau^{2}}\left(\frac{\tilde{R}_{f}-\tau^{2}\tilde{R}}{\frac{1}{\tau}+\frac{\partial\tilde{f}}{\partial t}}\right)^{*}e^{-\int_{\tau}^{t}\left(2\frac{\partial}{\partial s}\ln\left(\frac{1}{s}+\frac{\partial\tilde{f}}{\partial s}\right)+3\left(\frac{1}{s}+\frac{s}{s+1}\frac{\partial\tilde{f}}{\partial s}\right)\right)_{*}ds}d\tau.
\]
Applying the rescaling argument and the Schauder estimates to $\tilde{u}$, we can obtain an uniform bound 
$$
 \left\Vert \tilde{u} \right\Vert_{2+\alpha,I}<C,
$$
where $C$ is independent of $\tilde{u}$ and $I$ is a compact interval in $\mathbb{R}^+$. Applying Arzela-Ascoli Theorem, we thus obtain a solution $\tilde{u}^{-1}\in C^{2+\alpha}\left(A_{(0,\infty)}\right)$
bounded by 
\[
0<\tilde{\delta}_{*}\left(t\right)\leq\tilde{u}^{-2}\leq\tilde{\delta}^{*}\left(t\right).
\]
%%%%
It follows from the definition of the $C^{k+\alpha}$ norm that 
\[
\left\Vert \tilde{R}_f\right\Vert _{2+\alpha,\left[a,b\right]}\leq\left\Vert R_f\right\Vert _{2+\alpha,\left[a+1,b+1\right]}
\]
for any $0\leq a<b$, and we find that
\begin{eqnarray*}
\frac{\partial\tilde{f}}{\partial t} & \in & C^{2+\alpha}\left(A_{[0,\infty)}\right),\\
\tilde{R}_{f} & \in & C^{\alpha}\left(A_{[0,\infty)}\right),\\
t^{2}\tilde{R} & \in & C^{\alpha}\left(A_{[0,\infty)}\right).
\end{eqnarray*}
Since $0<1+\left(t+1\right)\frac{\partial f}{\partial t}\left(t+1\right)<\infty$
for all $t>0$, $  \left| \frac{t}{t+1}\frac{\partial\tilde{f}}{\partial t}\right| = \left| \frac{\partial f}{\partial t}(t+1) \right|<C$.
So $\tilde{\delta}_{*}\left(t\right)$ and $\tilde{\delta}^{*}\left(t\right)$
can be estimated on $\left(0,\epsilon\right),$ for some small $\epsilon>0$,
using (\ref{bdy est}): 
\[
(1-\eta)\leq\tilde{\delta}_{*}\left(t\right)<\tilde{\delta}^{*}\left(t\right)<\left(1-\eta\right)^{-1},
\]
which gives the bounds on $u^{-2}$ and $m$, and also shows that
for $0<t<\epsilon$, 
\[
-\frac{\eta t}{1-\eta}<2\tilde{m}\left(t\right)<\eta t,
\]
where $\tilde{m}\left(t\right)=\frac{t}{2}\left(1-\tilde{u}^{-2}\left(t\right)\right)=m\left(t+1\right)-1/2.$
The rescaling estimate (\ref{rescaling m}) applied to $\tilde{m}$ shows that the covariant
derivatives of $m$ decay,
\[
\left|\nabla m\left(t\right)\right|+\left|\nabla^{2}m\left(t\right)\right|\leq C\left(t-1\right),
\]
from which it follows that the curvature of $\bar{g}$ is bounded on $A_{\left[1,2\right]}.$

\end{proof}
%%%% Proof 2 ends.

%%%%
%%%%   Section 4 
%%%%
\section{Under Ricci Flow Foliation}

Let $\left(\Sigma,g_{1}\right)$ be a given 2-sphere with area $A( \Sigma )=4\pi$, and $N=[1,\infty)\times\Sigma$
equipped with the metric \[
\bar{g}=u^{2}dt^{2}+t^{2}g_{ij}(t,x)dx^{i}dx^{j},\]
where $g_{ij}(t,x)$ is the solution of the modified Ricci flow. Direct
computation shows that the metric $\bar{g}=u^{2}dt^{2}+t^{2}g(t,x)$
has the scalar curvature $\bar{R}$ if and only if $u$ satisfies the
parabolic equation (\ref{eq:RFu}) \[
t\frac{\partial u}{\partial t}=\frac{1}{2}u^{2}\Delta u+\frac{t^{2}}{4}\left|M\right|^{2}u+\frac{1}{2}u-\frac{1}{4}\left(R-t^{2}\bar{R}\right)u^{3},
\]
where $\Delta$ is the Laplacian with respect to $g$, $R$ and $\bar{R}$ are the scalar curvatures with respect to $g$ and $\bar{g}$ respectively, and $\left| M  \right|^2 = M_{ij}M_{kl}g^{ik}g^{jl}$.  
In this section we will prove existence results, Theorem \ref{thm:M3}
and Theorem \ref{thm:M4}.
%%% Lemma 10
\begin{lem}
\label{lem:bds}Suppose $u\in C^{2+\alpha}(A_{[t_{0},t_{1}]})$, $1\leq t_{0}<t_{1}$,
is a positive solution. If we further assume that $\bar{R}$ is defined
on $A_{[1,\infty)}$ such that the functions
\[
{\displaystyle \delta_{*}(t)}=\frac{1}{t}\int_{1}^{t}\left(\frac{R}{2}-\frac{\tau^{2}}{2}\bar{R}\right)_{*}(\tau)\exp\left(-\int_{\tau}^{t}\frac{s{|M|^{*}}^{2}}{2}ds\right)d\tau\]
and
\[
\delta^{*}(t)=\frac{1}{t}\int_{1}^{t}\left(\frac{R}{2}-\frac{\tau^{2}}{2}\bar{R}\right)^{*}(\tau)\exp\left(-\int_{\tau}^{t}\frac{s|M|_{*}^{2}}{2}ds\right)d\tau\]
are defined and finite for all $t\in[t_{0},\infty)$. 
Then for $t_{0}\leq t\leq t_{1}$, we have\begin{equation}
u^{-2}(t,x)\geq\delta_{*}(t)+\frac{t_{0}}{t}\left({u^{*}(t_{0})}^{-2}-\delta_{*}(t_{0})\right)\exp\left(-\int_{t_{0}}^{t}\frac{s|M|^{*2}}{2}ds\right)\label{eq:RFup}\end{equation}
and
\begin{equation}
u^{-2}(t,x)\leq\delta^{*}(t)+\frac{t_{0}}{t}\left({u_{*}(t_{0})}^{-2}-\delta^{*}(t_{0})\right)\exp\left(-\int_{t_{0}}^{t}\frac{s|M|_{*}^{2}}{2}ds\right).\label{eq:RFlb}\end{equation}
\end{lem}

\begin{proof}
Apply the maximum principle to (\ref{eq:RFw}) 
\begin{equation*}
                  t\partial_{t}w=\frac{1}{2w}\Delta w+\frac{3}{2} u \nabla u \cdot \nabla w-\left(\frac{t^{2}}{2}|M|^{2}+1\right)w+\frac{R}{2}-\frac{t^{2}}{2}\bar{R}.
     \end{equation*}
We have, at the maximum of $u(t,x)$,
\[
t\frac{dw_{*}}{dt}\geq-\left(\frac{t^{2}}{2}|M|^{2}+1\right)^{*}w_{*}+\left(\frac{R}{2}-\frac{t^{2}}{2}\bar{R}\right)_{*}.\]
Solving the associated ordinary differential equation,\begin{eqnarray*}
u^{-2} & \geq & \frac{1}{t}\int_{t_{0}}^{t}\left(\frac{R}{2}-\frac{\tau{}^{2}}{2}\bar{R}\right)_{*}\exp(-\int_{\tau}^{t}\frac{s{|M|^{*}}^{2}}{2}ds)d\tau\\
 &  & +\frac{1}{t}t_{0}w_{*}(t_{0})\exp\left(-\int_{t_{0}}^{t}\frac{s{|M|^{*}}^{2}}{2}ds\right)\\
 & = & \delta_{*}(t)+\frac{t_{0}}{t}\left(w_{*}(t_{0})-\delta_{*}(t_{0})\right)\exp\left(-\int_{t_{0}}^{t}\frac{s{|M|^{*}}^{2}}{2}ds\right).\end{eqnarray*}
Similarly, applying the maximum principle to $w^{*}$, we get the
upper bound of $u^{-2}.$ 
\end{proof}

%%% Proof of Lemma 10 ends.

Since $M$ converges to $0$ exponentially, we can obtain the interior
Schauder estimates. Let $I=[1,t_{1}]$ and $I'=[t_{0},t_{1}]$ with
$1 < t_{0}<t_{1}$. For a solution $u\in C^{2+\alpha}\left(A_{I}\right)$
with a source function $\bar{R}\in C^{\alpha}\left(A_{I}\right)$ and satisfying \[
0<\delta_{0}\leq u^{-2}\left(t,x\right)\leq\delta_{0}^{-1}\quad\hbox{for\, all\,}\left(t,x\right)\in A_{I},\]
for some constant $\delta_{0},$ we have
\begin{eqnarray}
\left\| u\right\| _{2+\alpha',I'}\leq C\left(t_{0},t_{1},\delta_{0},\left\| R\right\| _{\alpha,I},\left\| M\right\| _{\alpha,I},\left\| \bar{R}\right\| _{\alpha,I}\right), \label{RF int est}
\end{eqnarray}
for some $0<\alpha'<1,$ $\alpha'=\alpha'(\delta_{0}).$

\begin{thm}
\label{thm: existence}Assume that $\bar{R}\in C^{\alpha}(N)$ and the
constant $K$ is defined by \begin{equation}
K=\sup_{1\leq t<\infty}\left\{ -\int_{1}^{t}\left(\frac{R}{2}-\frac{\tau{}^{2}}{2}\bar{R}\right)_{*}\exp(\int_{1}^{\tau}\frac{s{|M|^{*}}^{2}}{2}ds)d\tau\right\} <\infty.\label{eq:K}\end{equation}
 Then for every $\varphi\in C^{2,\alpha}(\Sigma)$ such that \[
0<\varphi(x)<\frac{1}{\sqrt{K}},\quad\hbox{for\, all}\: x\in\Sigma,\]
there is a unique positive solution $u\in C^{2+\alpha}(N)$
of (\ref{eq:RFu}) with the initial condition\begin{equation}
u(1,\cdot)=\varphi(\cdot).\label{eq:ic}\end{equation}
\end{thm}
\begin{proof}
The upper bound of (\ref{eq:RFup}) implies that \[
\delta_{*}(t)+\frac{1}{t}(u^{*}(1))^{-2}\exp\left(-\int_{t_{0}}^{t}\frac{s{|M|^{*}}^{2}}{2}ds\right)>0\quad\hbox{for\, all\,}t\geq1.\]

By the short time existence of parabolic equations and the Schauder
estimates, there is $\epsilon>0$ and $u\in C^{2+\alpha}(A_{[1,1+\epsilon]})$
satisfying the initial condition $u(1)=\varphi$ on $[1,1+\epsilon]$
for some $\epsilon>0$. By Lemma \ref{lem:bds}, there are functions
$0<\delta_{1}(t)\leq\delta_{2}(t)<\infty$, $1\leq t$, independent
of $\epsilon$, such that \[
0<\delta_{1}(t)\leq u^{-2}(t,x)\leq\delta_{2}(t)\quad\mbox{for all}\quad t\in[1,1+\epsilon].\]
Let $U=\{t\in\mathbb{R}^{+}:\exists u\in C^{2+\alpha}(A_{[1,1+t]})\hbox{\, satisfying\,(\ref{eq:RFu})\, and\,(\ref{eq:ic})}\}$.
The short time existence guarantees $U$ is open in $\mathbb{R}^{+}$.
%Since $[1,1+t]$ is compact, there is $\delta_{0}$ such that $0<\delta_{0}\leq u^{-2}(t,x)\leq\delta_{0}^{-1}$ for all $(t,x)\in A_{[1,1+t]}$. 
From the interior estimate (\ref{RF int est})
, we have an a priori estimate for $||u(1+t,\cdot)||_{2,\alpha}$.
By the short time existence again, the solution can be extended to
$A_{[1,1+t+T]}$ for some $T$ independent of $u$, which shows that
$U$ is closed. Hence $u$ extends to a global solution $u\in C^{2+\alpha}(N)$.
\end{proof}

%%%Proof 11 ends.

%%%Corollary 12

\begin{cor}
Let $\bar{R}$ and $K$ be as given in Theorem \ref{thm: existence}
. Suppose $H\in C^{2,\alpha}(\Sigma)$ satisfies\[
H(x)>2\sqrt{K}\quad\hbox{for\, all\,}\, x\in\Sigma.\]
Then there is a metric $\bar{g}$ with scalar curvature $\bar{R}$ having
boundary $\Sigma_{1}$ with mean curvature $H$.\end{cor}
\begin{proof}
Let ${\displaystyle \varphi(x)=\frac{2}{H(x)}}$. Then the assumption\[
H(x)=\frac{2}{\varphi(x)}>2\sqrt{K}\quad\hbox{for\, all\,}\, x\in\Sigma\]
is equivalent to ${\displaystyle \varphi(x)<\frac{1}{\sqrt{K}}}$.
Theorem \ref{thm: existence} shows that there exists a unique solution
$u\in C^{2+\alpha}(N)$ to the initial value problem, and the resulting
metric has boundary $\Sigma$ with mean curvature $H$ by (\ref{eq:RFH}).
\end{proof}

\begin{thm}
\label{thm:M3-2}Let $u\in C^{2+\alpha}(N)$ be a solution of (\ref{eq:RFu}).
Suppose that $\bar{R}$ is given such that 
$${\displaystyle \int_{1}^{\infty}|\bar{R}|^{*}t^{2}dt<\infty},$$
and suppose
there is a constant $C>0$ such that for all $t\geq1$
$${\displaystyle }||\bar{R}t^{2}||_{\alpha,I_{t}}\leq\frac{C}{t},$$
where  $I_{t}=[t,4t]$.
Then $\bar{g}$ satisfies the asymptotically flat condition for $t>t_{0}$,
where $t_{0}$ is  some fixed constant. Moreover, the Riemannian
curvature $\bar{Rm}$ of the 3-metric $\bar{g}$ on $N$ is H\"{o}lder continuous
and decays as ${\displaystyle |\bar{Rm}|<\frac{C}{t^{3}}}$, and ADM
mass can be expressed as \begin{equation}
m_{ADM}=\lim_{t\rightarrow\infty}\frac{1}{4\pi}\oint_{\Sigma_{t}}\frac{t}{2}(1-u^{-2})d\sigma.\label{eq:ADM}\end{equation}
\end{thm}

\begin{proof}
We compare $\bar{g}$ with the flat metric $\dot{g}=dt^{2}+t^{2}\sigma$,
and get\[
\bar{g}-\dot{g}=(u^{2}-1)dt^{2}+t^{2}(g(t)-\sigma).\]
Since $g(t)$ converges to the round metric $\sigma$ exponentially,
to prove the asymptotic flatness of the metric $\bar{g}$, it suffices
to show that 
\[
|u^{2}-1|+|t\partial_{t}u|+|\nabla_{i}u|\leq\frac{C}{t},\quad i=1,2.
\]
Define $\tilde{u}(t,x)=u(\tau t,x)$. Observe that $u\in C^{2+\alpha}(A_{[\tau,4\tau]})$
satisfies (\ref{eq:RFu}) if and only if $\tilde{u}$ on the interval
$[1,4]$ satisfies \[
{\displaystyle t\frac{\partial\tilde{u}}{\partial t}=\frac{1}{2}\tilde{u}^{2}\Delta\tilde{u}+\frac{t^{2}}{4}|\tilde{M}|^{2}\tilde{u}+\frac{1}{2}\tilde{u}-\frac{\tilde{R}}{4}\tilde{u}^{3}+\frac{t^{2}\tilde{R}}{4}\tilde{u}^{3},}\]
where ${\displaystyle \tilde{M}(t,x)=\tau M(\tau t,x)}$, ${\displaystyle \tilde{R}(t,x)=R(\tau t,x)}$,
and ${\displaystyle \tilde{R}(t,x)=\tau^{2}\bar{R}(\tau t,x)}$.

Applying the Schauder estimates to $\tilde{u}$ on the interval $[1,4]$,
and then rescaling back, we have ${\displaystyle ||u||_{2+\alpha,I_{\tau}'}\leq C}$,
and
\begin{equation}
||m||_{2+\alpha,I_{\tau}'}\leq C||m||_{0,I_{\tau}}+{C\tau}\left(||\tau M||_{\alpha,I_{\tau}}+||R-2||_{\alpha,I_{\tau}}+||\tau^{2}\bar{R}||_{\alpha,I_{\tau}}\right),\label{ineq:m}
\end{equation}
where $I'_{\tau}=[2\tau,4\tau]$, $I_{\tau}=[\tau,4\tau]$, and $C$
is independent of $u$ and $\tau$.  $M$ converges to $0$, and $R$
converges to 2 exponentially fast under Ricci flow. By the assumption ${\displaystyle \int_{1}^{\infty}|\bar{R}|^{*}t^{2}dt<\infty}$, we have
\[
1-\frac{C}{t}\leq\delta_{*}(t)\leq\delta^{*}(t)\leq1+\frac{C}{t}\quad\hbox{for\, all\,}\, t\geq1,
\]
which controls $||m||_{0,I_{\tau}}$. 
Exponential convergence of $M$ and $R-2$, together with the decay assumption
${\displaystyle ||\bar{R}t^{2}||_{\alpha,I_{t}}\leq\frac{C}{t}}$, controls
the second term of (\ref{ineq:m}). Hence there is a uniform bound\[
||m||_{2+\alpha,I_{\tau}'}\leq C\quad\hbox{for\, all\,}\tau\geq1.\]
Expressing this in terms of $u$ and derivatives of $u$ gives \[
||1-u^{-2}||_{\alpha,I_{\tau}'}+||\tau\partial_{\tau}u||_{\alpha,I_{\tau}'}+||\nabla u||_{\alpha,I_{\tau}'}+||\nabla^{2}u||_{\alpha,I_{\tau}'}\leq\frac{C}{\tau}.\]
The estimates for $1-u^{-2}$ and $\nabla u$ show that $\bar{g}$ is
asymptotically Euclidean. The estimate for $\nabla^{2}u$, together
with the expression of the Ricci curvature of $\bar{g}$ in terms of
$u,$ shows that $\bar{Rc}\in C^{0,\alpha}(N)$ and ${\displaystyle \left|\bar{Rc}(t,x)\right|\leq\frac{C}{t^{3}}}$.
Thus, the Riemannian curvature on $N$ is H\"{o}lder continuous
and decays as ${\displaystyle |\bar{Rm}|<\frac{C}{t^{3}}}$.

Let $(\Omega,g)\hookrightarrow M$ be a compact three manifold with
smooth boundary $\Sigma$. The Hawking quasi local mass $m_{H}(\Sigma)$
is defined by \cite{Haw68}\[
m_{H}(\Sigma)=\sqrt{\frac{A(\Sigma)}{16\pi}}\left(1-\frac{1}{16\pi}\int H^{2}d\sigma\right).\]
Although $\Sigma_{t}=\{t\} \times \Sigma$ are not round spheres, the sequence $\Sigma_{t}$ approaches
the round sphere exponentially. Thus the Hawking mass $m_{H}(\Sigma_{t})$
approaches ADM mass of the asymptotically flat 3-manifold $N$ as
$t\rightarrow\infty$ (see \cite{SWW09}). Since the area of $\Sigma$
is normalized to $4\pi$ and the Ricci flow preserves the area, we
have the area ${\displaystyle A(\Sigma_{t})=4\pi t^{2}}$, the mean
curvature ${\displaystyle H=\frac{2}{tu}}$ for each leaf $\Sigma_{t}$.
The Hawking mass ${\displaystyle m_{H}(\Sigma_{t})}$ is given by
\begin{eqnarray*}
{\displaystyle m_{H}(\Sigma_{t})} & = & \sqrt{\frac{A(\Sigma_{t})}{16\pi}}\left(1-\frac{1}{16\pi}\int H_{t}^{2}d\sigma_{t}\right)\\
 & = & \frac{t}{2}\left(1-\frac{1}{16\pi}\int\left(\frac{4}{t^{2}u^{2}}\right)t^{2}d\sigma\right)\\
 & = & \frac{t}{2}\left(1-\frac{1}{4\pi}\int u^{-2}d\sigma\right)\\
 & = & \frac{1}{4\pi}\int\frac{t}{2}(1-u^{-2})d\sigma\end{eqnarray*}
 and then the ADM mass of $N$ is \[
m_{ADM}=\lim_{t\rightarrow\infty}\frac{1}{4\pi}\oint_{\Sigma_{t}}\frac{t}{2}(1-u^{-2})d\sigma.\]

\end{proof}
Theorems \ref{thm: existence} and \ref{thm:M3-2} together thus give
Theorem \ref{thm:M3}. Last, we prove Theorem \ref{thm:M4}.

\begin{thmM4} Let $\bar{R}\in C^{\alpha}(N).$ Suppose that $\bar{R}t^{2}<R$
for $1\leq t<\infty$. Let $0<\eta<1$ be such that\begin{equation}
1-\eta< R-\bar{R}|_{t=1}<(1-\eta)^{-1}.\label{bd R_N}\end{equation}
Then there is a solution $u^{-1}\in C^{2+\alpha}(N)$ such that the
constructed metric on $N$ has curvature uniformly bounded on $A_{[1,2]}$
with totally geodesic boundary $\Sigma$. \end{thmM4}
\begin{proof}
To prove Theorem \ref{thm:M4}, we define
\[
\tilde{u}(t)=\sqrt{\frac{t}{t+1}}u(t+1)\quad\hbox{where\,}\,\ensuremath{t\in(0,\infty)},\]
and obtain a parabolic equation for $\tilde{u}.$ Secondly, the equation
of $\tilde{u}$ has the same form as the $u$-equation\begin{equation}
t\frac{\partial\tilde{u}}{\partial t}(t)=\frac{1}{2}\tilde{u}^{2}\Delta\tilde{u}+\frac{t^{2}}{2}|\tilde{M}|^{2}\tilde{u}+\frac{1}{2}\tilde{u}-\frac{\tilde{R}}{4}\tilde{u}^{3}+\frac{t^{2}\tilde{R}_{N}}{4}\tilde{u}^{3},\label{eq:tilde}\end{equation}
where the fields $\tilde{M}$, $\tilde{R}$, and $\tilde{R}_{N}$
are defined by\begin{eqnarray*}
\left| \tilde{M}(t) \right|^2 & = & \frac{t+1}{t} \left| M(t+1) \right|^2,\\
\tilde{R}(t) & = & R(t+1),\\
t^{2}\tilde{R}_{N}(t) & = & (t+1)^{2}\bar{R}(t+1).\end{eqnarray*}
The equation for the corresponding term $\tilde{w}=\tilde{u}^{-2}$ is
     \begin{equation}\label{tilde{w} in thm 4}
                  t\partial_{t}\tilde{w}=\frac{1}{2\tilde{w}}\Delta \tilde{w}+\frac{3}{2} \tilde{u} \nabla \tilde{u} \cdot \nabla \tilde{w}-\left(\frac{t^{2}}{2}|\tilde{M}|^{2}+1\right)\tilde{w}+\frac{\tilde{R}}{2}-\frac{t^{2}}{2}\tilde{R}_N.  
       \end{equation}
The curvature assumption $\bar{R}t^{2}<R$
for $1\leq t<\infty$ implies that there exist functions 
\[
{\displaystyle \tilde{ \delta}_{*}(t)}=\frac{1}{t}\int_{0}^{t}\left(\frac{\tilde{R}}{2}-\frac{\tau^{2}}{2}\tilde{R}_N\right)_{*}(\tau)\exp\left(-\int_{\tau}^{t}\frac{s{|\tilde{M}|^{*}}^{2}}{2}ds\right)d\tau\]
and
\[
\tilde{\delta}^{*}(t)=\frac{1}{t}\int_{0}^{t}\left(\frac{\tilde{R}}{2}-\frac{\tau^{2}}{2}\tilde{R}_N\right)^{*}(\tau)\exp\left(-\int_{\tau}^{t}\frac{s|\tilde{M}|_{*}^{2}}{2}ds\right)d\tau\]
so that\[
0<\tilde{\delta}_{*}(t)\leq\tilde{\delta}^{*}(t)<\infty\quad\hbox{for\, all}\quad t>0.\]
Let $\ensuremath{\varphi_{\epsilon}\in C^{2,\alpha}(\Sigma)},\ensuremath{0<\epsilon<1}$
be any family of functions satisfying \[
\tilde{\delta}_{*}(\epsilon)\leq\varphi_{\epsilon}^{-2}(x)\leq\tilde{\delta}^{*}(\epsilon).\]
Applying the parabolic maximum principle to (\ref{tilde{w} in thm 4}) for $\tilde{w}=\tilde{u}^{-2}$, we can show that the solution $\tilde{u}_{\epsilon}$ exists for all time $[\epsilon,\infty)$
with initial condition $\varphi_{\epsilon}$ and \[
\tilde{\delta}_{*}(t)\leq\tilde{u}_{\epsilon}^{-2}(t,x)\leq\tilde{\delta}^{*}(t),\quad\epsilon\leq t<\infty,\]
for all $0<\epsilon<1$. 

Now suppose $t_{0}>0$ and $v\in C^{2+\alpha}(A_{I})$, $I=[t_{0},4t_{0}]$,
is a solution of (\ref{eq:tilde}) satisfying \[
\tilde{\delta}_{*}(t)\leq v^{-2}(t,x)\leq\tilde{\delta}^{*}(t),\quad\hbox{for\, all\,}\, t\in I,\]
and define $\tilde{v}(t,x)=v(t/t_{0},x)$. Applying the Schauder estimates
to $\tilde{v}$ on the interval $[1,4]$ and rescaling back, we obtain
an estimate of the form \begin{equation}
||v||_{2+\alpha,I'}\leq C,\quad I'=[2t_{0},4t_{0}],\label{ineq: v}\end{equation}
where $C$ is a constant independent of $v$. Applying (\ref{ineq: v})
to $\tilde{u}_{\epsilon}$ shows, by Arzela-Ascoli theorem, there
is a sequence $\epsilon_{j}\rightarrow0$ such that $\{\tilde{u}_{\epsilon_{j}}\}$
converges uniformly in $C^{2+\alpha}(A_{I})$ for any compact interval
$I\subset\mathbb{R}^{+}$ to the desired solution $\tilde{u}\in C^{2+\alpha}(A_{(0,\infty)}).$ 

Moreover, since $|M|\leq Ce^{-ct}$ for some constants $c$ and $C$,
and (\ref{bd R_N}), there is a small $\epsilon$ such that on $(0,\epsilon)$,  $1-\eta < \tilde{R}-t^{2}\tilde{R}_N<(1-\eta)^{-1}$,
\begin{eqnarray*}
\frac{1}{t}\int_{0}^{t}\exp\left(-\int_{\tau}^{t}\frac{s{|\tilde{M}|^{*}}^{2}}{2}ds\right)d\tau & < & 1,
\end{eqnarray*}
and
\begin{eqnarray*}
\frac{1}{t}\int_{0}^{t}\exp\left(-\int_{\tau}^{t}\frac{s{|\tilde{M}|_{*}}^{2}}{2}ds\right)d\tau & < & 1.
\end{eqnarray*}
It shows that on $(0,\epsilon)$\begin{eqnarray*}
1-\eta\leq\tilde{\delta}_{*}(t) & \leq & \tilde{\delta}^{*}(t)<(1-\eta)^{-1},
\end{eqnarray*}
and
\begin{eqnarray*}
\frac{-\eta t}{1-\eta} & < & 2\tilde{m}(t)<\eta t,
\end{eqnarray*}
where ${\displaystyle \tilde{m}(t)=\frac{t}{2}(1-\tilde{u}^{-2}(t))=m(1+t)-\frac{1}{2}}$.
The rescaling Schauder estimate shows that the covariant derivatives
of $m$ decay, \[
|\nabla m(t)|+|\nabla^{2}m(t)|\leq C(t-1).\]
It follows that the curvature of $\bar{g}$ is bounded on $A_{[1,2]}$. \end{proof}
%%% Proof 4 ends. 

%%%Corollary 14

\begin{cor}
\label{cor:schwarzchild}If we start with the standard metric $(\Sigma,\sigma)$
and prescribe the scalar curvature $\bar{R}\equiv0$, then the metric
$\bar{g}$ obtained from above is exactly a Schwarzschild metric with
ADM mass ${\displaystyle m_{ADM}=\frac{1}{2}}$. \end{cor}
\begin{proof}
Since the initial metric is the standard round metric, the modified
Ricci flow doesn't change the metric and $g(t)=\sigma
$, $R\equiv2$,
and $M_{ij}\equiv0$ for all $t\geq1$. Since $\bar{R}\equiv0$, the
rescaling fields are \begin{eqnarray*}
{\displaystyle \tilde{R}=2,\quad\tilde{M}_{ij}\equiv0,\quad\hbox{and}\quad\tilde{R}_{N}\equiv0}\end{eqnarray*}
 for all $t\geq0$. Moreover the a priori bounds are $\tilde{\delta}_{*}(t)=1$
and $\tilde{\delta}^{*}(t)=1$ for all $t\geq0$. Hence we have the solution
$\tilde{u}\equiv1$ and the metric \begin{eqnarray*}
{\displaystyle \bar{g}=\frac{1}{1-\frac{1}{t}}dt^{2}+t^{2}g_{S^{2}}.}\end{eqnarray*}

\end{proof}

\thanks{
I would like to thank Professor Mu-Tao Wang and  Maria Gordina, for their patience and discussions. I also want to extend my gratitude to Panagiota Daskalopoulos, and Richard Hamilton for their time and generosity. I would also like to express my sincere appreciation to Lan-Hsuan Huang, Michael Eichmair, and Ye-Kai Wang for their encouragement.  %and remove the percent signs.
}

\bibliographystyle{spmpsci}      % mathematics and physical sciences
\bibliography{references}   % name your BibTeX data base

\end{document}